\documentclass[11pt,oneside,english,reqno]{amsart}
\usepackage[T1]{fontenc}
\usepackage[latin9]{inputenc}
\usepackage[a4paper]{geometry}
\usepackage{mathrsfs}
\usepackage{amstext}
\usepackage{amsthm}
\usepackage{amssymb}
\usepackage{color}
\usepackage{dsfont}
\usepackage{enumerate}
\usepackage{verbatim} 
\usepackage{enumitem}
\usepackage{hyperref}

\makeatletter
\numberwithin{equation}{section}
\numberwithin{figure}{section}
\theoremstyle{plain}
\newtheorem{thm}{\protect\theoremname}[section]
\theoremstyle{definition}
\newtheorem{defn}[thm]{\protect\definitionname}
\theoremstyle{remark}
\newtheorem{rem}[thm]{\protect\remarkname}
\theoremstyle{plain}
\newtheorem{lem}[thm]{\protect\lemmaname}
\theoremstyle{plain}
\newtheorem{prop}[thm]{\protect\propositionname}
\theoremstyle{plain}

\theoremstyle{plain}
\newtheorem{cor}[thm]{\protect\corollaryname}
\theoremstyle{plain}

\theoremstyle{plain}

\theoremstyle{plain}
\newtheorem*{thm*}{\protect\theoremname}

\makeatother

\usepackage{babel}
\providecommand{\corollaryname}{Corollary}
\providecommand{\definitionname}{Definition}
\providecommand{\lemmaname}{Lemma}
\providecommand{\propositionname}{Proposition}
\providecommand{\conjecturename}{Conjecture}
\providecommand{\remarkname}{Remark}
\providecommand{\theoremname}{Theorem}
\providecommand{\hypothesisname}{Hypothesis}
\providecommand{\examplename}{Example}

\newcommand{\cA}{\mathcal{A}}

\newcommand{\cC}{\mathcal{C}}

\newcommand{\cF}{\mathcal{F}}

\newcommand{\cI}{\mathcal{I}}

\newcommand{\cN}{\mathcal{N}}

\newcommand{\cP}{\mathcal{P}}

\newcommand{\BB}{\mathbb{B}}

\newcommand{\EE}{\mathbb{E}}

\newcommand{\NN}{\mathbb{N}}

\newcommand{\PP}{\mathbb{P}}

\newcommand{\RR}{\mathbb{R}}

\begin{document}

\title{$C^{\infty}-$ regularization  of ODEs perturbed by noise}

\author{Fabian Andsem Harang}

\address{FH: Department of Mathematics, University of Oslo,\\ Postboks 1053 Blindern\\
Oslo, 0316,
Norway\\
fabianah@math.uio.no}

\author{Nicolas Perkowski}

\address{NP: Institut f\"ur Mathematik, Freie Universit\"at Berlin, \\ Arnimallee 7\\
14195 Berlin,
Germany\\
perkowski@math.fu-berlin.de}

\begin{abstract}
We study ODEs with vector fields given by general Schwartz distributions, and we show that if we perturb such an equation by adding an ``infinitely regularizing'' path, then it has a unique solution and it induces an infinitely smooth flow of diffeomorphisms. We also introduce a criterion under which the sample paths of a Gaussian process are infinitely regularizing, and we present two processes which satisfy our criterion. The results are based on the path-wise
space-time regularity properties of local times, and solutions are constructed using the approach of Catellier-Gubinelli based on non-linear
Young integrals. 
\end{abstract}

\thanks{
{\em Acknowledgments} We are grateful to the referee for his very careful reading of the paper, and for pointing out the observation and computation in Remark~\ref{rem:weaker LND}. F.H. gratefully acknowledge funding from the Research Council of Norway (RCN) project: STORM, project number: 274410. Part of the work was carried out while N.P. was employed at Max-Planck-Institut f\"ur Mathematik in den Naturwissenschaften, Leipzig. N.P. gratefully acknowledges financial support by DFG through the Heisenberg program.}

\maketitle

\section{Introduction and main results }

The regularizing effect of adding an irregular stochastic process
to an ill-posed ordinary differential equations (ODE) has been extensively studied over the last fifty years. Still it is one of the most surprising results at the intersection of analysis and probability
theory. Consider the integral version of an ODE under perturbation 
of a path $w:[0,T]\rightarrow \mathbb{R}^d$ for some $\epsilon\in \mathbb{R}$  given by 
\begin{equation}
y_{t}^{x}=x+\int_{0}^{t}b\left(y_{r}\right)dr+\epsilon w_{t},\qquad x\in \mathbb{R}^d. \label{eq: ODE perturbation}
\end{equation}
 If $\epsilon=0$, then the classical theory of ODEs would essentially require local Lipschitz
continuity of the vector field $b$ to obtain the uniqueness of solutions. However, for $\epsilon\neq 0$ and suitable $w$
one can show the existence of a unique solution under more general assumptions on the vector field $b$. This has been studied in a number of papers, e.g. \cite{Zvonkin1974, Veretennikov1981, Bass2001,Krylov2005,Meyer-Brandis2010,Menoukeu2013,NilProBan,Catellier2016}.
In recent years particular interest has been directed towards the
regularizing properties of a fractional Brownian motion with Hurst
parameter $H\in\left(0,1\right)$. It has been proven, both by probabilistic
means, e.g. in \cite{NilProBan}, and by path-wise analysis in \cite{Catellier2016},
that the lower we choose $H$, the more general assumptions we may choose on $b$. In \cite{Catellier2016} Catellier and Gubinelli show that uniqueness may hold for Equation~\ref{eq: ODE perturbation} even if $b$ is only a distribution. More precisely, they show that if $w$ is a fractional Brownian motion with Hurst index $H$, then for all $b\in B_{\infty,\infty}^{\alpha}$ with $\alpha>1-\frac{1}{2H}$, where $B_{\infty,\infty}^{\alpha}$
is a Besov space of regularity $\alpha$, the solution to (\ref{eq: ODE perturbation})
almost surely exists uniquely. The null set outside of which the uniqueness fails depends on the initial condition $x$, the noise $w$,  and on the vector field $b \in  B^\alpha_{\infty, \infty}$. Under the stronger regularity assumption  $\alpha > 2 - \frac{1}{2H}$, they show that the
flow $x\mapsto y_{t}^{x}$ is Lipschitz, but even then the null set may depend on the drift coefficient\footnote{Theorem 1.14 in \cite{Catellier2016} in fact claims that $\alpha>\frac{3}{2}-\frac{1}{2H}$ is sufficient to have a Lipschitz flow. However, as pointed out by the referee, there seems to be no full proof of this claim, at least not in the case $\alpha \le 0$. Corollary 2.20 of \cite{Catellier2016} would require $\alpha>2-\frac{1}{2H}$, and we have therefore chosen to cite this requirement to be on the safe side.}. Catellier and Gubinelli also identify a path-wise condition for $w$ under which uniqueness holds for all sufficiently regular $b$ (measured in terms of ``Fourier-Lebesgue regularity'' rather than Besov regularity) and all initial conditions $x$, see \cite[Theorem~1.14]{Catellier2016}.
\\

In the more recent work \cite{ProskeBanosAmine} the authors consider an infinite sequence of fractional Brownian motions
$\left(\lambda_{k}w^{H_{k}}\right)_{k\geq0},$ where $\left(\lambda_{k}\right)_{k\geq0}$ and $\left(H_{k}\right)_{k\geq0}$
are suitable null sequences, and
\begin{equation}\label{proske process}
\mathbb{B}_{t}:=\sum_{k\geq0}\lambda_{k}w_{t}^{H_{k}}.
\end{equation}
Using techniques developed in \cite{NilProBan}, they show that the equation 
\[
y_{t}^{x}=x+\int_{0}^{t}b\left(r,y_{r}\right)dr+\mathbb{B}_{t}
\]
has a unique strong solution (in the probabilistic sense) as long
as $b\in L^{p}\left(\left[0,T\right],L^{q}\left(\mathbb{R}^{d},\mathbb{R}^{d}\right)\right)\cap L^{1}\left(\left[0,T\right],L^{\infty}\left(\mathbb{R}^{d},\mathbb{R}^{d}\right)\right)$,
and furthermore they show that the flow map $x\mapsto y_{\cdot}^{x}$ is
in $C^{\infty}$. The techniques used to obtain this results are mainly
based on Malliavin calculus and probabilistic methods, and again the null set outside of which the results fail might depend on $b$. \\

In the current article we unite the two perspectives of \cite{Catellier2016} and \cite{ProskeBanosAmine} and we provide a general framework to obtain existence
and uniqueness as well as differentiability of the flow
map associated to Equation (\ref{eq: ODE perturbation}) for some
sufficiently irregular paths $w$. In contrast to both \cite{Catellier2016} and \cite{ProskeBanosAmine}, our analysis is purely
path-wise.

Similarly as in \cite{Catellier2016} we formulate the equations in  the framework of non-linear Young theory. But rather than considering directly the regularity of the random map $(t,x) \mapsto \int_0^t b(x+w_s) ds$, which can only be controlled outside of a null set that depends on $b$, we first control the regularity of the local time $L$ of $w$, and then write $\int_0^t b(x+w_s) ds = b\ast (L(-\cdot))$. In this way the null set is independent of $b$. There is also such a purely path-wise result in \cite{Catellier2016}, but the regularity of $L$ is given in a Fourier-Lebesgue space and therefore in applications also $b$ has to be in a suitable Fourier-Lebesgue space -- or we need to apply embedding results to derive the regularity of $L$ in Sobolev spaces. Here we work with more common function spaces (Sobolev spaces rather than Fourier-Lebesgue spaces), which has the advantage that we get path-wise regularizing effects directly for $b$ in Hölder spaces or even $L^2$ Sobolev spaces. A second advantage is that our analysis applies to any regularizing path, and not only stochastic paths. That is, given a path with a sufficiently regular local time, existence and uniqueness of ODEs of the form \eqref{eq: ODE perturbation} is readily obtained.

 To this end, we identify a class of Gaussian processes with exceptional regularizing properties, and we use the space-time regularity of their local times.  In fact, if
$w:\left[0,T\right]\rightarrow\mathbb{R}^{d}$ is a 
Gaussian process with co-variance function satisfying some simple conditions (translating roughly speaking to sufficient irregularity), then its local time $L:\Omega\times\left[0,T\right]\times\mathbb{R}^{d}\rightarrow\mathbb{R}$ almost surely
has infinitely many derivatives in its spatial variable for
almost all $\omega\in\Omega$. So to analyze the ODE~\eqref{eq: ODE perturbation} we assume that $w$ is a fixed path with smooth local time. By definition of the local time, we have say for bounded measurable $b$ and $x\in \mathbb{R}^d$ 
\begin{equation}\label{eq:first int by parts}
\int_{s}^{t}b\left(x-w_{r}\left(\omega\right)\right)dr=\left[b\ast L_{s,t}(\omega)\right](x),
\end{equation}
where $f_{s,t} := f_t - f_s$ for any function $f$, and $\ast$ denotes convolution. So for regular $L_{s,t}$ we can make sense of $\int_s^t \nabla b(x-w_r(\omega)) dr = [b\ast \nabla L_{s,t}(\omega)](x)$, even if $b$ is not differentiable.
 This observation allows us to obtain bounds for integrals appearing in (\ref{eq: ODE perturbation}) which only depend on low regularity norms of $b$: consider, for convenience of notation, Equation~\eqref{eq: ODE perturbation} with $\epsilon = - 1$. Then  $\tilde{y}^x_t=y^x_t+w_t$ solves
\begin{equation}\label{eq:modi ODE}
\tilde{y}^x_t =x +\int_{0}^t b(\tilde{y}^x_r-w_r)dr,
\end{equation}
and the integral term on the right hand side is very similar to the one on the left hand side of (\ref{eq:first int by parts}). In fact, the integral $\int_{0}^t b(\tilde{y}^x_r-w_r)dr$ can formally be interpreted as 
\begin{equation}\label{eq:abstr avg op}
\int_{0}^t b(\tilde{y}^x_r-w_r)dr=\int_{0}^t \left[b\ast L_{dr}\right](\tilde{y}^x_r). 
\end{equation}
For now this expression is purely formal as we need to make sense of the differential $L_{dr}$, which later we will do via non-linear Young integration (giving a simplified derivation of results from~\cite{Catellier2016}). 
\\

 We are mainly interested in ``infinitely regularizing'' paths $w$.  In this case, we will show that the solution $y^x$ to (\ref{eq: ODE perturbation})
exists uniquely (up to a possibly finite explosion time), and the
flow $x\mapsto y^x$ is $C^{\infty}$, under the sole assumption that $b \in \mathscr S^\prime$ is a Schwartz
distribution. 
Let us first specify what we mean by an infinitely regularizing path:

\begin{defn}\label{def:infinitely-regularizing}
We say that a measurable path $w:[0,T]\rightarrow\mathbb{R}^d$ is \emph{infinitely  regularizing} if the local time $L:[0,T]\times \mathbb{R}^d\rightarrow\mathbb{R}$, defined in Section \ref{subsec:Occupation-measures-and},  is in $C_{T}^{\gamma}\mathcal{C}^\alpha$  for all $\gamma\in(\frac{1}{2},1)$ and $\alpha\in \mathbb{R}$.
\end{defn}

In the above definition, and throughout the text,  the parameter $T>0$ will always be finite, and  the space $C_{T}^{\gamma}\mathcal{C}^\alpha:=C^\gamma([0,T],\mathcal{C}^\alpha(\mathbb{R}^d))$ denotes  the space of H\"older continuous functions $h:[0,T]\rightarrow \mathcal{C}^\alpha (\mathbb{R}^d )$ with values in the Besov space $\mathcal{C}^\alpha(\mathbb{R}^d):= B^\alpha_{\infty,\infty}$. More details on these spaces can be found in Section \ref{beso spaces}. 

Our first main result is that existence and uniqueness hold for ODEs perturbed by the path $w$, with drift coefficients given by general Schwartz distributions in $\mathscr{S}^\prime$. Moreover,  the flow mapping  $x\mapsto y^x_\cdot$ is infinitely differentiable.

\begin{thm}\label{thm:main}
Let $b\in\mathscr{S}^{\prime}$ be a Schwartz
distribution, and consider a continuous infinitely regularizing path $w:\left[0,T\right]\rightarrow\mathbb{R}^{d}$
as in Definition \ref{def:infinitely-regularizing}. 
Then for all $x \in \mathbb R^d$ there exists $T^\ast = T^\ast(x) \in (0,T] \cup \{\infty\}$ such that there is a unique solution to the equation 
\[
y_{t}^{x}=x+\int_{0}^{t}b\left(y_{r}^x\right)dr+w_{t},
\]
 in $C\left(\left[0,T^\ast\right) \cap [0,T],\mathbb{R}^{d}\right)$, interpreted in the sense of Definition~\ref{def:perturbed-ode}. For $T^\ast(x) < \infty$ we have $\lim_{t \uparrow T^\ast(x)} |y^x_t| = \infty$. Moreover, the map $x \mapsto T^\ast(x)^{-1}$ is locally bounded, and if $\tau < T^\ast(x)$ for all $x \in U$ for an open set $U$, then the flow mapping $U \ni x \mapsto y^x_\cdot \in C([0,\tau],\mathbb R^d)$ is infinitely Fr\'echet differentiable.

\end{thm}

\begin{prop}
	Assume in the setting of Theorem~\ref{thm:main} that additionally $b \in B^\alpha_{p,q}$ for some $\alpha \in \mathbb R$ and $p,q \in [1,\infty]$. Then $T^\ast(x) = \infty$ for all $x \in \mathbb R^d$.
\end{prop}

It should be noted that all this holds for deterministic paths that are infinitely regularizing. However, the derivation of sharp spatio-temporal regularity results for the local times of deterministic functions (for example the Weierstrass function) is still an open and challenging problem,  although some progress has recently been made in this direction, see \cite{imkeller2020rough}.
Therefore, we show that there exist infinitely regularizing stochastic processes. In particular, we will prove the following theorem, outlining sufficient conditions for a Gaussian process to be infinitely regularizing.

\begin{thm}\label{thm: Cond for gaussian regualrizing}
Let $w:[0,T]\times\Omega\rightarrow \mathbb{R}^d$ be a centered Gaussian process on a complete filtered probability space $(\Omega,\cF,\{\cF_t\}_{t\in [0,T]},\PP)$. Suppose $w$ satisfies the following local non-determinism condition for any $\zeta\in (0,1)$
\begin{equation*}
    \inf_{t \in (0,T]}\inf_{s\in [0,t)} \inf_{\substack{z\in \RR^d: \\|z|=1}}  \frac{ z^T \mathrm{cov}(w_t|\cF_s)z}{(t-s)^{2\zeta}}>0,
\end{equation*}
where $\mathrm{cov}(w_t|\cF_s) := \EE[(w_t - \EE[w_t|\cF_s])(w_t - \EE[w_t|\cF_s])^T|\cF_s]$.
Then for almost all $\omega\in \Omega$ the path $t\mapsto w_t(\omega)$  is infinitely regularizing.
These conditions are satisfied by the log-Brownian motion of Definition \ref{def:log BM}, or the process $\mathbb{B}_t := \sum_k \lambda_k w^{H_k}_t$  in~\eqref{proske process}.
\end{thm}

\begin{rem}
	As pointed out by the referee, actually the weaker condition	
	\begin{equation*}
    	\inf_{t \in (0,T]}\inf_{s\in [0,t)} \inf_{\substack{z\in \RR^d: \\|z|=1}}  \frac{ z^T \mathrm{cov}(w_t-w_s)z}{(t-s)^{2\zeta}}>0,\qquad \zeta \in (0,1)
	\end{equation*}
	is already sufficient to guarantee infinite regularization. See the discussion in Remark~\ref{rem:weaker LND}.
\end{rem}

The structure of the paper is as follows:
\begin{itemize}[leftmargin=1cm]

 \item In Section \ref{Sec: Loc times and besov spaces} we  provide some background material on local times and Besov spaces, as well as a statement of the stochastic sewing lemma, recently developed by L\^e \cite{le2018}. 

\item Section \ref{sec:Gaussian regualrizers}
 is devoted to the proof of  Theorem \ref{thm: Cond for gaussian regualrizing}.

\item In Section \ref{sec: Infinitely regularizing stochastic processes} we present two infinitely regularizing Gaussian processes. In particular, we will consider a process  called  $p-$log Brownian motion, and show that is infinitely regularizing. We also show that the processes used in \cite{ProskeBanosAmine} is infinitely regularizing. 

\item In Section \ref{sec:regularzing operators} 
we give a path-wise construction of the so called averaging operators, in line with what has been done
in \cite{Catellier2016}. We show that when these operators are constructed from an infinitely regularizing path, then they are in  $C^{\gamma}_T\mathcal{C}^{\alpha}$
for any $\gamma\in\left(0,1\right)$, $\alpha>0$.

\item Finally, in Section \ref{sec:ex a Uni} we prove the existence and uniqueness of solutions to (\ref{eq: ODE perturbation}) and the smoothness of the associated flow.
\end{itemize}

\section{Essentials of local times and Besov spaces}\label{Sec: Loc times and besov spaces}

\subsection{Occupation measures and Local times \label{subsec:Occupation-measures-and}}

The occupation measure of an $\mathbb{R}^{d}$-valued measurable path $w:[0,T]\rightarrow \mathbb{R}^d$
at a time $t\in\left[0,T\right]$ is defined by
\[
\mu_{t}\left(A\right):=\lambda\left\{ s\in\left[0,t\right]|w_{s}\in A\right\}, \qquad \text{for }A\in\mathcal{B}(\mathbb{R}^{d}),
\]
where $\lambda$ is the Lebesgue measure. We interpret $\mu_t(A)$ as ``the amount of time 
$w$ spends in $A$ up to time $t$''. Occupation measures have been an important topic in the theory of stochastic processes during the last fifty years. We refer the interested reader to the comprehensive review paper by
Geman and Horowitz \cite{GemHoro}, and the references therein for an introductory account of occupation measures and local times. 
\begin{defn}
\label{def:Defintion-of-local time}Let $w:[0,T]\rightarrow \mathbb{R}^d$
be a measurable path. Assume that there exists a measurable function $L:\left[0,T\right]\times\mathbb{R}^{d}\rightarrow\mathbb{R}_{+}$
with $L_{0}\left(z\right)=0$ for all $z\in\mathbb{R}^{d}$ and such that
\begin{equation}
\mu_t\left(A\right) =\int_{A}L_{t}\left(z\right)dz,\qquad A \in \mathcal B( \mathbb R^d), t \in [0,T].\label{eq:local time rel}
\end{equation}
 Then we call $L$ the \emph{local time}
of $w$. 
\end{defn}

\begin{rem}
Of course, the local time does not have to exist, and intuitively we interpret its existence and regularity as an irregularity condition for $w$. Later we will see that for sample paths of Gaussian processes this interpretation is in some sense justified. Already in \cite[Section 4]{Berman69} it was observed that the regularity of the local time is connected to the irregular behavior of $w$.  In general it has been a long standing  open problem to establish a clear link between regularity properties of the local times and irregularity measures such as ``true roughness'' \cite{Hairer2013Regularity, Friz2013}; see Sections~10 and~11 of \cite{GemHoro} for some partial results in that direction. However, in a very recent preprint, the authors of \cite{GaleatiGubinelli2020} are able to establish one direction of this connection: In Corollary 68 they show that any function with a local time of appropriate (space-time) regularity must be ``H\"older rough'' in a certain sense. The converse direction, i.e. if any H\"older rough function has a nice local time, is still open to the best of our knowledge.
Note that in $d=1$ even $w\in C^\infty$ can have a local time: For $a,b \in \mathbb R$ with $b > 0$ the path $w_t = a + b t$ has the local time $L_t(z) = b^{-1} \mathds{1}_{(a,a+bt]}(z)$. However, if a Lipschitz continuous $w \in C([0,T], \mathbb R)$ has a local time $L$, then $L_t$ has at least two discontinuities: If $z_0 = w_{t_0} = \max\{w_s: s \in [0,t]\}$, then $L_t(z) = 0$ for $z > z_0$, and with the Lipschitz constant $K$ of $w$ we get $w_s \in [z_0-\delta,z_0]$ for all $s \in [t_0-\delta/K,t_0+\delta/K] \cap [0,t]$. For $\delta>0$ small enough we must have $t_0 - \delta/K \ge 0$ or $t_0 + \delta/K \le t$, and therefore
\[
	\int_{z_0-\delta}^{z_0} L_t(z) dz \ge \frac{\delta}{K}.
\]
By the fundamental theorem of calculus this is impossible if $L_t$ is continuous with $L_t(z_0) = 0$. Similarly $L_t$ must have a discontinuity at the minimum of $w|_{[0,t]}$. In other words the local time can only be continuous if $w$ is more irregular than Lipschitz continuous. A similar argument shows that the set of paths with local time $L_T \in C^\alpha(\mathbb R)$ has empty intersection with $C^\beta([0,T],\mathbb R)$ for all $\beta>(\alpha+1)^{-1}$.
\end{rem}

\subsection{Essentials of Besov spaces}\label{beso spaces}

Here we recall some basic properties of Besov
spaces. For a more extensive introduction we refer to \cite{BahCheDan}. We will denote by $\mathscr{S}$ resp.
$\mathscr{S}^{\prime}$ the space of Schwartz functions on $\mathbb R^d$ resp. its dual, the space of tempered distributions. For $f\in\mathscr{S}^{\prime}$
we denote the Fourier transform by $\hat{f}=\mathscr{F}\left(f\right) = \int_{\mathbb R^d}e^{-i x \cdot} f(x)  dx$, where the integral notation is formal, with inverse $\mathscr{F}^{-1} f = (2\pi)^{-d}\int_{\mathbb R^d} e^{i z \cdot} \hat f(z)dz$.
\begin{defn}
\label{def:Dyadic partition of unity}Let $\chi,\rho\in C^{\infty}(\mathbb{R}^{d},\mathbb{R})$
be two radial functions such that $\chi$ is supported on a ball $\mathcal{B}=\left\{ |x|\leq c\right\} $
and $\rho$ is supported on an annulus $\mathcal{A}=\left\{ a\leq|x|\leq b\right\} $
for $a,b,c>0$, such that 
\begin{align*}
\chi+\sum_{j\geq0}\rho\left(2^{-j}\cdot\right) & \equiv1,\\
{\rm supp}\left(\chi\right)\cap{\rm supp}\left(\rho\left(2^{-j}\cdot\right)\right) & =\emptyset,\,\,\,\forall j\ge1,\\
{\rm supp}\left(\rho\left(2^{-j}\cdot\right)\right)\cap{\rm supp}\left(\rho\left(2^{-i}\cdot\right)\right) & =\emptyset,\,\,\,\forall|i-j|\geq1.
\end{align*}
Then we call the pair $\left(\chi,\rho\right)$ a \emph{dyadic partition
of unity}. Furthermore, we write $\rho_{j}=\rho(2^{-j} \cdot)$
for $j\geq 0$ and $\rho_{-1}=\chi$, as well as $K_{j}=\mathscr{F}^{-1}\rho_{j}$.
\end{defn}

The existence of a partition of unity is shown for example in \cite{BahCheDan}. We fix a partition of unity $(\chi,\rho)$ for the rest of the paper.

\begin{defn}
\label{def:Paley littlewood block}For $f\in\mathcal{\mathscr{S}}^{\prime}$
we define its Littlewood-Paley blocks by
\[
\Delta_{j}f=\mathscr{F}^{-1}(\rho_{j}\hat{f}) = K_j \ast f.
\]
It follows that $f=\sum_{j\geq-1}\Delta_{j}f$ with convergence in $\mathscr S'$.
\end{defn}

In the following we write
\begin{equation}
	\langle x \rangle := (1+|x|^2)^{1/2}.	
\end{equation}

\begin{defn}
\label{def:Besov space}For any $\alpha,\kappa\in\mathbb{R}$ and $p,q\in\left[1,\infty\right]$,
the \emph{weighted Besov space} $B_{p,q}^{\alpha}(\langle x \rangle^{\kappa})$ is
\[
B_{p,q}^{\alpha}(\langle x \rangle^\kappa):=\left\{ f\in\mathscr{S}^{\prime}\left|\|f\|_{B_{p,q}^{\alpha}(\langle x \rangle^\kappa)}=\left(\sum_{j\geq-1}\left(2^{j\alpha}\| \langle x \rangle^\kappa \Delta_{j}f \|_{L^{p}}\right)^{q}\right)^{\frac{1}{q}}<\infty\right.\right\},
\]
with the usual interpretation as $\ell^\infty$ norm if $q = \infty$. If $\kappa=0$, we simply write $B^\alpha_{p,q}$. Furthermore, we denote $\mathcal{C}^{\alpha}(\langle x \rangle^\kappa)=B_{\infty,\infty}^{\alpha}(\langle x \rangle^\kappa)$
and $\mathcal{C}_{p}^{\alpha}(\langle x \rangle^\kappa)=B_{p,\infty}^{\alpha}(\langle x \rangle^\kappa)$.
\end{defn}

\begin{rem}
	By Theorem~6.5 of \cite{Triebel2006} we have
	\[
		\|f\|_{B^\alpha_{p,q}(\langle x \rangle^\kappa)} \simeq \|f \langle x \rangle^\kappa \|_{B^\alpha_{p,q}}.
	\]
\end{rem}

\begin{rem}
For $\alpha\in \mathbb R_+ \setminus \mathbb N$ the space $\mathcal{C}^{\alpha}(\langle x \rangle^\kappa)$
corresponds to a classical weighted H\"older space, see e.g. \cite[Lemma 2.1.23]{Martin2018Refinements}. For all $\alpha \in \mathbb R$ the Besov space $B_{2,2}^{\alpha}$ corresponds to the inhomogeneous
Sobolev space $H^{\alpha}$ defined by 
\[
H^{\alpha}:=\left\{ f\in\mathscr{S}^{\prime}\left|\|f\|_{H^{\alpha}}=\|\left(1+|\cdot|\right)^{\alpha}\hat{f}\|_{L^{2}} < \infty \right.\right\} .
\]
\end{rem}

\begin{lem}[Besov embedding, see \cite{BahCheDan}, Proposition 2.71]\label{lem:(Besov-embedding)} Let $1\leq p_{1}\leq p_{2}\leq\infty$
and $1\leq q_{1}\leq q_{2}\leq\infty,$ and let $\kappa\in\mathbb{R}$.
Then $B_{p_{1},q_{1}}^{\kappa}$ is continuously embedded
into $B_{p_{2},q_{2}}^{\kappa-d\left(\frac{1}{p_{1}}-\frac{1}{p_{2}}\right)}$. 
\end{lem}

Recall from Definition \ref{def:infinitely-regularizing} that a path is infinitely regularizing if its local time is in $C^\gamma_T \mathcal C^\alpha$ for all $\gamma \in (0,1)$ and all $\alpha \in \mathbb R$. By an interpolation argument this follows from a softer criterion:

\begin{cor}\label{cor:cont in time of LT}
	Let $w \in C([0,T], \mathbb R^d)$ with associated local time $L$ such that
	\[
		\sup_{t \in [0,T]} \| L_t \|_{H^\alpha} < \infty
	\]
	for any $\alpha > 0$. Then $w$ is infinitely regularizing.
\end{cor}

\begin{proof}
	We get the necessary time regularity by bounding $L_{s,t}$ in a Besov space that contains measures, and then we use an interpolation argument: for any finite positive measure $\mu$ we have by \cite[Proposition~2.76]{BahCheDan}
	\[
		\|\mu\|_{B^{0}_{1,\infty}} \lesssim \sup_{\substack{\varphi \in \mathscr S,\\ \|\varphi\|_{B^0_{\infty,1}} \le 1}} \langle \mu, \varphi\rangle \le \mu(\mathbb R^d) \times \sup_{\substack{\varphi \in \mathscr S, \\ \|\varphi\|_{B^0_{\infty,1}} \le 1}} \| \varphi\|_{L^\infty} \lesssim \mu(\mathbb R^d).
	\]
	Since the occupation measure is given by $\mu_t(\cdot)=\int_0 ^t \delta(\cdot-w_r)dr$,	with the Besov embedding result from Lemma \ref{lem:(Besov-embedding)} we get that 
	\[
		\|L_{s,t}\|_{\mathcal C^{-d}} \lesssim \|L_{s,t}\|_{B^0_{1,\infty}} \le \int_s^t \|\delta(\cdot - w_r)\|_{B^0_{1,\infty}} dr \lesssim |t-s|,
	\]
	which implies in particular that $L \in C^1_T\mathcal{C}^{-d}(\mathbb{R}^d)$. 
     Now we get for  $\alpha > 0$ and  $\gamma \in (0,1)$:
     \begin{equation*}
\begin{aligned}
		\|L\|_{C^\gamma_T \mathcal{C}^\alpha} &= \sup_{0\le s < t \le T} \sup_{j \ge -1} 2^{j\alpha} \frac{\| \Delta_j L_{s,t} \|_{L^\infty}}{|t-s|^\gamma}
		\\
		&\le \left( \sup_{s,t,j} 2^{-jd} \frac{\| \Delta_j L_{s,t} \|_{L^\infty}}{|t-s|} \right)^\gamma \left( \sup_{s,t,j} 2^{j \frac{\alpha+\gamma d}{1-\gamma}} \| \Delta_j L_{s,t} \|_{L^\infty} \right)^{1-\gamma}
		\\
		& \le \|L\|_{C^1_T\mathcal{C}^{-d}}^\gamma \sup_{s,t} \|L_{s,t}\|^{1-\gamma}_{\mathcal{C}^\kappa},
\end{aligned}
\end{equation*}
for $\kappa:={\frac{\alpha+\gamma d}{1-\gamma}}$. 
 Since  $\|L_{s,t}\|_{\mathcal C^{\kappa}} \lesssim \| L_{s,t} \|_{H^{\kappa + d/2}} \lesssim 1$ by assumption, our claim follows. 
\end{proof}

\subsection{The stochastic sewing lemma}

To derive the space-time regularity of local times, we will apply the  stochastic sewing lemma \cite{le2018} recently developed by Khoa L\^{e}. We therefore recite here the statement of this lemma, and refer the reader to \cite{le2018} for the proof and a discussion of this result. We will make use of $n$- simplices defined by 
\begin{equation}\label{simplex}
    \Delta_n^T:=\{(t_1,\ldots,t_n)\in [0,T]^n|0\leq  t_1\leq \ldots\leq t_n\leq T\}. 
\end{equation}

\begin{lem}[\cite{le2018}, Theorem~2.1]\label{stochastic sewing lemma}
Let $(\Omega,\mathcal{F},\{\mathcal{F}_t\}_{t\in[0,T]},\mathbb{P})$ be a complete filtered probability space. Let $p\geq 2$ and let $A:\Delta_2^T\rightarrow \mathbb{R}^d$ be a stochastic process such that $A_{s,s}=0$,  $A_{s,t}$ is $\mathcal{F}_t$ measurable, and $(s,t)\mapsto A_{s,t}$ is right-continuous from  $\Delta_2^T$ into $L^p(\Omega)$. Set $\delta_u A_{s,t}:=A_{s,t}-A_{s,u}-A_{u,t}$ for $(s,u,t)\in \Delta_3^T$, and assume that there exists constants $\beta>1$, $\kappa>\frac{1}{2}$, and $C_1,C_2>0$ such that 
\begin{equation}\label{eq:integrand cond}
\begin{aligned}
        \|\mathbb{E}\left[\delta_u A_{s,t}|\mathcal{F}_s\right]\|_{L^p(\Omega)} &\leq K_1 |t-s|^\beta,
        \\
        \|\delta_u A_{s,t}\|_{L^p(\Omega)} &\leq K_2 |t-s|^\kappa.       
\end{aligned}
\end{equation}
Then there exists a unique (up to modifications) $\{\mathcal{F}_t\}$-adapted stochastic process $\mathcal{A}$ such that the following properties are satisfied: 
\begin{itemize}[leftmargin=1cm]
    \item[\rm (i)] $\mathcal{A}:[0,T]\rightarrow L^p(\Omega)$ is right continuous, and $\mathcal{A}_0=0$. 
    \item[\rm (ii)] There exist constants $C_1,C_2>0$ such that for $\cA_{s,t}=\cA_t-\cA_s$:
    \begin{equation}\label{bounds on stochastic integral}
        \begin{aligned}
             \|\mathcal{A}_{s,t}-A_{s,t}\|_{L^p(\Omega)}&\leq  C_1 K_1 |t-s|^\beta+ C_2 K_2 |t-s|^\kappa,
             \\
             \|\EE\left[\mathcal{A}_{s,t}-A_{s,t}|\cF_s\right]\|_{L^p(\Omega)}& \leq C_1 K_1 |t-s|^\beta.
        \end{aligned}
    \end{equation}
\end{itemize}
Furthermore, for all $(s,t)\in \Delta_2^T$ and for any partition $\cP$ of $[s,t]$, define
\begin{equation}
    A^{\cP}_{s,t}:=\sum_{[u,v]\in \cP}A_{u,v}. 
\end{equation}
Then $A^{\cP}_{s,t}$ converge to $\cA_{s,t}$ in $L^p(\Omega)$ as the mesh size $|\cP|\rightarrow 0$. 
\end{lem}

\section{Regularity of local times associated to  Gaussian paths}\label{sec:Gaussian regualrizers}

Here we study the space-time regularity of the local times of Gaussian processes.
Although there are well known results for the spatial regularity of the local time $L_t$ of Gaussian processes at fixed times, or regularity in time for fixed spatial variable  (e.g. \cite{GemHoro}), it seems more difficult to find results that quantify the  joint space-time regularity (see however \cite{Catellier2016} for results about the time regularity of the local time of fractional Brownian motion in certain Fourier-Lebesgue spaces, and also~\cite{Xiao1997} for some results in that direction). We therefore present a general criterion for centred Gaussian processes to be infinitely regularizing.

\begin{defn}\label{def: Gaussian}

A square-integrable Gaussian process $w:\Omega \times [0,T] \to \RR^d$ is called \emph{$\zeta$-locally non-deterministic ($\zeta$-LND)} if
\begin{equation}\label{eq:LND condition}
    \inf_{t \in (0,T]}\inf_{s\in [0,t)} \inf_{\substack{z\in \RR^d: \\|z|=1}}  \frac{ z^T \mathrm{cov}(w_t|\cF_s)z}{(t-s)^{2\zeta}}>0,
\end{equation}
where $\mathrm{cov}(w_t|\cF_s) := \EE[(w_t - \EE[w_t|\cF_s])(w_t - \EE[w_t|\cF_s])^T|\cF_s]$.

\end{defn}

\begin{rem}\label{rem:deterministic cond variance}
 Note that any stochastic processes $w$  can be decomposed into two parts, namely 
\begin{equation*}
w_r=\EE[w_r|\cF_s]+\left(w_r-\EE[w_r|\cF_s]\right).
\end{equation*}
For a Gaussian process $w$, the components  $\EE[w_r|\cF_s]$ and $\left(w_r-\EE[w_r|\cF_s]\right)$ are two Gaussian random variables such that $\left(w_r-\EE[w_r|\cF_s]\right)$ is independent of $\mathcal F_s$, see e.g. \cite[Theorem 3.10.1]{Bogachev1998}. This implies that the conditional covariance $\mathrm{cov}\left(w_r\right|\cF_s)$ is deterministic, due to the fact that  $\mathrm{cov}(w_r|\cF_s) = \mathrm{cov}\left(w_r-\EE[w_r|\cF_s]\right|\cF_s)$ and $w_r-\EE[w_r|\cF_s]$ is independent of $\mathcal F_s$.  
\end{rem}

\begin{rem}
There exists several different definitions of the concept of local non-determinism of stochastic process, e.g. \cite{Berman1969, Pitt78, Xiao2006}. The condition in Definition~\ref{def: Gaussian} is related to the strong local $\phi$ non-determinism  proposed by Cuzick and DuPreez \cite{CuzikDuPerez1982}, where $\phi(r)=r^\gamma$; the only (important) difference is that we only condition on the past, while in \cite{CuzikDuPerez1982} also information about the future is taken into account. This concept is also discussed in \cite{Xiao2006}. 
\end{rem}

The next theorem shows that if $\zeta\in [0,\frac{1}{d})$, then $\zeta$-locally non-deterministic centered Gaussian processes have jointly H\"older-Sobolev continuous local times.

\begin{thm}\label{thm: regualritry of avg op}
Let $\zeta\in [0,\frac{1}{d})$ and let $w:\Omega \times [0,T]\rightarrow \mathbb{R}^d$ be a continuous centered Gaussian  process which is $\zeta$-LND. 
 Then there exists a null set $\mathcal{N} \subset \Omega$ 
 such that  for all $\omega\in \cN^c$ the function $w(\omega)$ has a local time $L(\omega)$, and for all $\lambda <\frac{1}{2\zeta}-\frac{d}{2}$ and $\gamma\in [0,1-(\lambda+\frac{d}{2})\zeta)$ we have
\begin{equation}\label{eq: Local time reg}
    \|L_{s,t}(\omega)\|_{H^{\lambda}}\leq C(\omega) |t-s|^\gamma.
\end{equation}
It follows that $L\in C^\gamma_T H^\lambda$, $\PP$-almost surely.
\end{thm}

\begin{proof}
We control the Sobolev regularity by deriving bounds for the Fourier transform of the occupation measure. By definition, the Fourier transform of $\mu_{s,t}$ is given by $\widehat{\mu_{s,t}}(z)=\int_s^t e^{i\langle z, w_r\rangle }dr$. 
Consider the process $A_{s,t}:=
\int_s^t \EE[e^{i \langle z,w_r\rangle}|\cF_s]dr$ (we will deal with the $z$ dependence of $A$ later, but for now we suppress the notation), where $\{\cF_t\}$ is the (completion of the) natural filtration generated by $w$. We will apply the stochastic sewing lemma to derive bounds for the moments of the limit $\mathcal A_{s,t}$ of the Riemann sums $\sum_{[u,v]\in \cP} A_{u,v}$. Then we will see that in fact $\mathcal A_{s,t} = \widehat \mu_{s,t}(z)$.

To apply the stochastic sewing lemma, we need to check that $A$ verifies the necessary conditions. By definition of $A$, it follows directly that $A_{s,s}=0$, $A_{s,t}$ is $\cF_t$ measurable, and $(s,t)\rightarrow A_{s,t}$ is right continuous. 
Furthermore,
\begin{equation*}
    \EE[\delta_u A_{s,t} |\cF_s]
    =\EE\left[\int_s^t \EE[e^{i\langle z,w_r\rangle}|\cF_s]dr-\int_s^u \EE[e^{i\langle z,w_r\rangle}|\cF_s]dr-\int_u^t \EE[e^{i\langle z,w_r\rangle}|\cF_u]dr|\cF_s\right]=0,
\end{equation*}
by the tower property of conditional expectations, and thus the condition $\|\EE[\delta_u A_{s,t}|\cF_s]\|_{L^p(\Omega)}
    =0$ in \eqref{eq:integrand cond} is satisfied.  
To show the second condition, i.e. $\| \delta_u A_{s,t} \|_{L^p(\Omega)}\leq C_1 |t-s|^{\kappa}$ for some $\kappa>\frac12$, we invoke the fact that the conditional covariance of a Gaussian process is deterministic as shown in Remark \ref{rem:deterministic cond variance}. Consequently,
\begin{equation*}
    \EE[e^{i\langle z,w_r\rangle}|\cF_s]=\exp\left(i\langle z,\mu_{r}^{\cF_s}\rangle -\frac{1}{2}z^T \Sigma^{\cF_s}_r z \right),
\end{equation*}
where $\mu_r^{\cF_s}:=\EE[w_r|\cF_s]$ and $\Sigma^{\cF_s}_r:=\mathrm{cov}(w_r|\cF_s)$. This yields
\begin{align*}
    \|\delta_u A_{s,t}\|_{L^p(\Omega)} & = \left\| \int_u^t \exp\left(i\langle z,\mu_{r}^{\cF_s}\rangle -\frac{1}{2}z^T \Sigma^{\cF_s}_r z \right)-\exp\left(i\langle z,\mu_{r}^{\cF_u}\rangle -\frac{1}{2}z^T \Sigma^{\cF_u}_r z \right)dr \right\| \\
    &\lesssim \int_u^t  \left[\exp\left( -\frac{1}{2}z^T \Sigma^{\cF_s}_r z \right)+\exp\left( -\frac{1}{2}z^T \Sigma^{\cF_u}_r z \right)\right]dr. 
\end{align*}

By assumption, $w$ is $\zeta$-LND, so denote by $M$ the constant given by the left hand side of \eqref{eq:LND condition}. Since  $(r-s)^{2\zeta}\geq (r-u)^{2\zeta}$ for any $(s,u)\in \Delta_2^T$,  we observe that
\begin{equation*}
      \| \delta_u A_{s,t} \|_{L^p(\Omega)} \lesssim  \int_u^t \exp\left( -\frac{M}{2}|z|^2(r-u)^{2\zeta} \right)dr. 
\end{equation*}
It is readily checked that 
\begin{equation*}
e^{-\frac{M}{2}(r-u)^{2\zeta}|z|^2}\leq e^{\frac{MT^{2\zeta}}{2}}e^{-\frac{M}{2}(r-u)^{2\zeta}(1+|z|^2)},    
\end{equation*}
and that for $\lambda' \geq 0$ we have  $e^{-C}\lesssim C^{-\frac{\lambda'}{2}}$, uniformly in $C>0$.
So we get for $\lambda'>0$ such that 
$\lambda' \zeta < 1$: 
\begin{align*}
      \| \delta_u A_{s,t} \|_{L^p(\Omega)} &\lesssim \frac{M^{\lambda'}}{2^{\lambda'}}e^{\frac{MT^{2\zeta}}{2}} \int_u^t (1+|z|^{2})^{-\frac{\lambda^\prime}{2}}(r-u)^{-\lambda^\prime\zeta} z \,dr 
      \\
      &\simeq (1+|z|^{2})^{-\frac{\lambda^\prime}{2}}(t-u)^{1-\lambda^\prime\zeta} \\
      & = (1+|z|^{2})^{-\frac{\lambda^\prime}{2}}(t-u)^{\kappa}.
\end{align*}

If $\lambda^\prime < \frac{1}{2\zeta}$, then $1-\lambda'\zeta > \frac12$ and we can apply the stochastic sewing lemma, more precisely~\eqref{bounds on stochastic integral} together with Minkowski's inequality, to deduce that the ``sewing'' $\mathcal A_{s,t}$ satisfies
\begin{equation*}
    \|\mathcal A_{s,t}\|_{L^p(\Omega)}\lesssim \|A_{s,t}\|_{L^p(\Omega)}+  (1+|z|^2)^{-\frac{\lambda^\prime}{2}}|t-s|^{1-\lambda^\prime\zeta},
\end{equation*} 
where we recall that $    \|\EE[\delta_u A_{s,t}|\cF_s]\|_{L^p(\Omega)}=0$. 
It is now readily seen, following the lines of the previous analysis,  that we also have  
\begin{equation*}
    \|A_{s,t}\|_{L^p(\Omega)} \lesssim   (1+|z|^2)^{-\frac{\lambda^\prime}{2}}|t-s|^{1-\lambda^\prime\zeta}.
\end{equation*}
Moreover, we get for $t^n_k = s + (t-s)k/n$
\begin{align*}
    \| \widehat{\mu_{s,t}}(z) - \mathcal A_{s,t}\|_{L^p(\Omega)} & \le \sum_{k=0}^{n-1} \left\|\int_{t^n_k}^{t^n_{k+1}} (e^{i \langle z, w_r \rangle} - \EE[e^{i \langle z, w_r \rangle}|\mathcal F_{t^n_k}])dr \right\|_{L^p(\Omega)} \\
    & \le 2 \sum_{k=0}^{n-1} \int_{t^n_k}^{t^n_{k+1}} \left\| e^{i \langle z, w_r \rangle} - e^{i \langle z, w_{t^n_k}\rangle}\right\|_{L^p(\Omega)} dr,
\end{align*}
and since $w$ is continuous, the dominated convergence theorem shows that the right hand side converges to zero as $n\to\infty$. In conclusion we have shown that 
\begin{equation*}
    \|\widehat{\mu_{s,t}}(z)\|_{L^p(\Omega)}\lesssim   (1+|z|^2)^{-\frac{\lambda^\prime}{2}}|t-s|^{1-\lambda^\prime\gamma}.
\end{equation*}
We will now use this moment bound together with Kolmogorov's continuity criterion to derive the claimed regularity of $\mu$. 
For $p\geq 2$ and $\epsilon \in (0, \lambda^\prime-\frac{d}{2})$ we apply Minkowski's inequality to obtain
\begin{align*}
    \EE[\|\mu_{s,t}\|_{H^{\lambda^\prime-\frac{d}{2}-\epsilon}}^p]^{\frac{1}{p}} &= \EE\left[\left(\int_{\mathbb{R}^d}|\widehat{\mu_{s,t}}(z)|^2(1+|z|^2)^{\lambda^\prime-\frac{d}{2}-\epsilon}dz\right)^{\frac{p}{2}}\right]^\frac{1}{p}
    \\
    &\leq \left(\int_{\mathbb{R}^d} \||\widehat{\mu_{s,t}}(z)|^2\|_{L^\frac{p}{2}(\Omega)}(1+|z|^2)^{\lambda^\prime-\frac{d}{2}-\epsilon}dz\right)^\frac{1}{2}
    \\
    &= \left(\int_{\mathbb{R}^d} \|\widehat{\mu_{s,t}}(z)\|^2_{L^p(\Omega)}(1+|z|^2)^{\lambda^\prime-\frac{d}{2}-\epsilon}dz\right)^\frac{1}{2}
    \\
    & \lesssim  \left(\int_{\mathbb{R}^d} \left( (1+|z|^2)^{-\frac{\lambda^\prime}{2}}|t-s|^{1-\lambda^\prime\zeta}\right)^2(1+|z|^2)^{\lambda^\prime-\frac{d}{2}-\epsilon}dz\right)^\frac{1}{2}
    \\
    &\lesssim |t-s|^{1-\lambda^\prime\zeta}\int_{\mathbb{R}^d} (1+|z|^2)^{-\frac{d}{2}-\epsilon}dz,
\end{align*}
and the integral on the right hand side is finite for any $\epsilon>0$. 
Since $p\geq 2$ can be chosen arbitrarily large, it follows from Kolmogorov's continuity theorem that for any $\gamma \in [0,1-\lambda^\prime\zeta)$ there exists a set $\cN^c$ of full measure such that for all $\omega\in \cN^c $ and $(s,t)\in \Delta_2^T$ we have  
\begin{equation}
    \|\mu_{s,t}(\omega)\|_{H^{\lambda^\prime-\frac{d}{2}-\epsilon}} \leq C(\omega) |t-s|^\gamma. 
\end{equation}
So with $\lambda=\lambda^\prime-\frac{d}{2}-\epsilon$ and we obtain the claimed result \eqref{eq: Local time reg}. Moreover, since $\zeta < \frac{1}{d}$ we can choose $\lambda > 0$ and in particular $\mu_t(\omega) \in L^2$ and the density $L_t(\omega)$ exists.
\end{proof}

\begin{rem}\label{rem:weaker LND}
As pointed out by the referee, the infinite regularization property of a Gaussian process could be obtained under a weaker local non-determinism condition: Suppose $w$ is a $d$-dimensional Gaussian process and that for some $\xi\in (0,1)$ there exists a $c_\xi>0$ such that 
\begin{equation}
\mathrm{cov}(w_t-w_s)\ge c_\xi|t-s|^{2\xi} I_{d},\quad \forall s,t\in [0,T],
\end{equation}
where the inequality is interpreted in the sense of quadratic forms. By similar computations as in the proof of Theorem \ref{thm: regualritry of avg op} we see that for $0\leq s<t\leq T$ with $|t-s|\leq 1$ 
\begin{align*}
\EE[\|\mu_{s,t}\|_{H^\lambda}^2]&\leq \int_{[s, t]^2} \int_{\mathbb{R}^d} (1 + | z |^2)^{\lambda} e^{- c_\xi | z
  |^2 (t - s)^{2 \xi}} d z d r d u
  \\
  &= | t - s |^2 \int_{\mathbb{R}^d} (1 + | z |^2)^{\lambda} e^{- c_\xi | z
  |^2 (t - s)^{2 \xi}} d z
  \\
  &= | t - s |^2 \int_{\mathbb{R}^d} (1 + | z |^2)^{\lambda} \frac{(t -
  s)^{- \xi d}}{(t - s)^{- \xi d}} e^{- c_\xi \frac{| z |^2}{(t - s)^{- 2 \xi}}}
  d z.
\end{align*}
Thus for a Gaussian random variable $Z\sim \cN(0,I_{d})$ we have for any $\lambda >-\frac d2$
\begin{equation}
\EE[\|\mu_{s,t}\|_{H^\lambda}^2]\leq |t-s|^{2-\xi d}\EE[(1+|(t-s)^{-\xi}Z|^2)^\lambda] \lesssim |t-s|^{2-\xi d-2\xi \lambda}. 
\end{equation}
Under the assumption that $\frac{1}{2\xi}-\frac{d}{2} >\lambda $ it follows from Kolmogorov's continuity criterion that $\mu \in C_T^{\frac12 - \frac{\xi d}{2} - \xi \lambda}H^\lambda$, and in particular $\sup_{t\in [0,T]} \|\mu_t\|_{H^\lambda}<\infty$. If the parameter $\xi$ can be chosen arbitrarily in $(0,1)$, this holds for all $\lambda$ and thus it follows from Corollary \ref{cor:cont in time of LT} that $w$ is infinitely regularizing. But Theorem~\ref{thm: regualritry of avg op} gives much better time regularity (with the previous argument we can never obtain time regularity better than $1/2$), and it controls the $L^p$-moments of the Sobolev norm.%
\end{rem}

\begin{rem}
Let $b\in C(\RR^d)$, and  $w:\Omega\times [0,T]\rightarrow \RR^d$ be a $\zeta$-LND Gaussian process for some $\zeta\in (0,\frac1d)$. Set $T^w_tb(x):=\int_0^tb(x+w_r)dr$, and observe that $T^w_tb=  b \ast (L_t(-\cdot))$ where $L_t$ is the local time of $w$. Invoking the regularity of the local time obtained in Theorem \ref{thm: regualritry of avg op} together with Young's convolution inequality, there exists a null set $\cN\subset \Omega$ only depending on $w$, such that for all $\omega\in \cN^c$  and for all $\epsilon>0$:
\begin{equation}
    \|T^{w(\omega)}_tb-T^{w(\omega)}_sb\|_{\cC^{\alpha+\frac{1}{2\zeta}-\frac{d}{2}-\epsilon}} \lesssim \|b\|_{H^\alpha} \| L_{s,t}(\omega) \|_{H^{\frac{1}{2\zeta}-\frac{d}{2}-\epsilon}} \lesssim \|b\|_{H^\alpha} |t-s|^\gamma,
\end{equation}
for some $\gamma>\frac{1}{2}$. Compared to Theorem 1.1 of \cite{Catellier2016} we lose $\frac{d}{2}$ derivatives in our estimate, but we gain integrability. The main difference is that the null set $\cN\subset \Omega$ in \cite{Catellier2016} depends on the function $b$. On  the other hand, through functional embeddings of the form $\cF L^{\rho,\infty}\hookrightarrow H^{\rho-\frac{d}{2}}$ one could deduce from the  Fourier-Lebesgue regularity estimates in \cite{Catellier2016} that if $w$ is a fractional Brownian motion, then the averaged field $T^{-w}\delta$ is contained in $H^{\frac{1}{2\zeta}-\frac{d}{2}-\epsilon}$. Noting that the local time $L$ is in fact equal to the averaged field $T^{-w}\delta$, one obtains the same regularity as in Theorem \ref{thm: regualritry of avg op}. One advantage of Theorem \ref{thm: regualritry of avg op} is that the argument is directly given in a classical Sobolev space, without need for functional embeddings. Also, the core of the argument might  be useful for estimating the space-time regularity in more general $B_{p,p}^\alpha$ Sobolev spaces, although the computations would become more involved than in the $L^2$ case which can be elegantly handled with Fourier arguments. It would also  be possible to directly estimate the regularity of $T^wb$ using similar arguments as in the proof of Theorem \ref{thm: regualritry of avg op}. Indeed, we observe that the Fourier transform of $T^wb$ is $\hat{b}(z)\int_s^t e^{i\langle z,w_r\rangle}dr$. It is then readily checked that we recover similar regularity results as in \cite[Theorem 1.1]{Catellier2016}, although in $H^\alpha$ spaces. But as the main goal of this article is to provide a path-wise analysis of infinitely regularizing paths, we want to avoid the dependence of the null sets on $b$ and therefore we estimate the regularity of $L$.
\end{rem}

\begin{rem}
	At least in the case of a fractional Brownian motion with Hurst parameter $\zeta$ we get from 
	\cite[Theorem~1.4]{Catellier2016} a control of the $\frac{1}{2\zeta}-$ regularity of $L$ in a Fourier-Lebesgue space, while here we only control the $\frac{1}{2\zeta}-\frac{d}{2}-$ Sobolev regularity. Implicitly, Conjecture~1.2 of~\cite{Catellier2016} suspects that the loss of $\frac{d}{2}$ derivatives in our result can be avoided and that we should have $L \in C^{\frac12+}B_{1,1}^{\frac{1}{2\zeta}-}$. Indeed, the  conjecture claims that the fractional Brownian motion $w$ satisfies for any Schwartz function $K$ with $\int_{\RR^d} K(x) dx =0$:
	\[
		\EE\left[\left(\int_{\RR^d}\left| \int_0^t K(x+w(s)) ds\right| dx\right)^p\right] \lesssim t^{p/2}, \qquad t \to \infty.
	\]
	Recall that the Littlewood-Paley blocks satisfy $\Delta_jf = K_j\ast f$, where $K_j(x)=2^{jd} K_0(2^jx)$ for $j\geq0$ and $K_0$ is a Schwartz function with $\int_{\RR^d}K_0(x) dx = 0$.
	By invoking the definition of the local time, the scaling invariance of the fractional Brownian motion, and the conjecture, we would obtain
	\begin{align*}
		\EE\left[\left| \int_{\RR^d} |\Delta_j L_t(x) |dx\right|^p\right] &= \EE\left[\left|\int_{\RR^d} \bigg| \int_{0}^t 2^{jd} K_0(2^{j}x-2^jw(s))ds\bigg|dx\right|^p\right]
	\\
	&= \EE\left[\left| \int_{\RR^d} 2^{jd}\bigg| \int_{0}^t K_0(2^{j}x-w(2^{j/\zeta}s))ds\bigg|dx\right|^p\right]
	\\
	&= 2^{-jp/ \zeta} \EE\left[\left| \int_{\RR^d} \bigg|
  \int_0^{2^{j / \zeta} t} K_0 (x - w (s)) ds \bigg| d x\right|^p\right]\\
    & \overset{\text{Conj.}}{\lesssim} 2^{-jp / \zeta}(2^{j/\zeta}t)^{p/2} = t^{p/2} 2^{-jp/ (2\zeta)}.
	\end{align*}
	From this we would obtain
	\begin{align*}
 	 \| \| L_t \|_{B_{1, 1}^{\alpha}} \|_{L^p (\Omega)}  & \leq  \sum_{j
 	 \geqslant -1} 2^{j \alpha} \left\| \int_{\RR^d} \left| \Delta_j L_t(x) \right| d x \right\|_{L^p (\Omega)} \\
 	 & \lesssim \left\| \int_{\RR^d} \left| \Delta_{-1} L_t(x) \right| d x \right\|_{L^p (\Omega)} + \sum_{j\ge0} 2^{j(\alpha - \frac{1}{2\zeta})}t^{1/2},
	\end{align*}
	and the sum in $j \ge 0$ is finite as long as $\alpha < \frac{1}{2\zeta}$, while the first term on the right hand side is easily shown to be finite. This would yield $L_t \in B^{\frac{1}{2\zeta}-\epsilon}_{1,1}$, 
and in particular $L_t \in L^1$ whenever $\frac{1}{2\zeta} > 0$. 
But of course in general it does not only depend on the Hurst parameter $\zeta$ but also on the dimension whether the fractional Brownian motion has an absolutely continuous occupation measure. For example, Xiao~\cite[Theorem~2.1]{Xiao1995} shows that if $w$ is a $d$-dimensional fractional Brownian motion of Hurst index $\zeta$, then the Hausdorff dimension of $(w_t)_{t \in [0,1]}$ is equal to $\min\{d, \frac{1}{\zeta}\}$. If $\frac{1}{\zeta} < d$,  the image of $(w_t)_{t \in [0,1]}$ is thus a null set in $\RR^d$ and therefore the occupation measure cannot be absolutely continuous.
This then also gives a negative answer to Conjecture 1.2 posed in \cite{Catellier2016}, at least whenever $\frac{1}{\zeta}<d$.
 Note also that $\frac{1}{\zeta} < d$ is equivalent to $\frac{1}{2\zeta} - \frac{d}{2} < 0$ and that Theorem~\ref{thm: regualritry of avg op} gives us space regularity $\frac{1}{2\zeta} - \frac{d}{2}-$, i.e. for $\frac{1}{2\zeta} - \frac{d}{2} > 0$ it follows from Theorem~\ref{thm: regualritry of avg op} that the local time exists (and then immediately has $L^2$-Sobolev regularity and not just $B_{1,1}$-Besov regularity).
\end{rem}

\section{Infinitely regularizing stochastic processes}\label{sec: Infinitely regularizing stochastic processes}

It follows from Theorem \ref{thm: Cond for gaussian regualrizing} together with Corollary \ref{cor:cont in time of LT} that if $w$ is a continuous centered Gaussian process which is $\zeta$-LND for any $\zeta > 0$, then $w$ is almost surely infinitely regularizing in the sense of Definition \ref{def:infinitely-regularizing}. Here we present two examples of such processes.

\subsection{$p-$log-Brownian motions}

If the conditional variance $\mathrm{Var}(w_{t+h}|\mathcal{F}_t)$ of a continuous centered Gaussian process is bounded below by $\phi(h):=|\ln(1/h)|^{-p}$, for some $p > 0$, then it is $\zeta$-LND for any $\zeta>0$. Thus our first example has an incremental variance structure resembling $\phi$. This is partly inspired by \cite{GemHoro}, where the authors mention in a remark below Theorem~28.4 that Gaussian processes with incremental variance behaving like the logarithm around the origin, i.e.  $\sim|\ln\left(1/t\right)|^{-1}$
for $t \to 0$, seem to have local times with exceptional (spatial) regularity. In \cite{MOCIOALCA_Viens2005} the authors investigate a Gaussian process they call the log-Brownian motion. The same process has also recently been investigated for the purpose of super rough volatility modelling in \cite{gulisashvili2020}.

\begin{defn}
\label{def:log BM}Consider $[0,T]\subset [0,1)$ and a $p>\frac{1}{2}$, and let $B:[0,T]\times \Omega \rightarrow \RR^d$ be a $d$-dimensional Brownian motion. We define
the \emph{$p-$log Brownian motion} as
\begin{equation}
w_{t}^{p}:=\int_{0}^{t}k(t-s)dB_{s},\label{eq:log bm}
\end{equation}
where $k(t):=|t\ln(1/t)^{2p}|^{-\frac{1}{2}} \in L^2([0,T])$, and the integration is to be understood component-wise.
\end{defn}

\begin{rem}
Since for $p>\frac{1}{2}$ the function $t^{-1}\ln(1/t)^{-2p}$ has a non-integrable singularity at $t=1$, we have to take $T<1$. For larger $T$ we could rescale the kernel and consider $k_\beta(t)=k(\beta t)$ for $\beta > T$ instead. See also the discussion below Definition 18 of~\cite{MOCIOALCA_Viens2005} or \cite{gulisashvili2020}. To obtain a stationary version we could for example consider $k(t)=(t(|\ln(1/t)^{2p}| \vee 1))^{-\frac{1}{2}} \in L^2(\RR_+)$  and then $w^p_t = \int_{-\infty}^t k(t-s) dB_s$ for a two-sided $d$-dimensional Brownian motion $B$. For simplicity we do not make these adaptations and we restrict to $T<1$ for the rest of the subsection.
\end{rem}

\begin{prop}\label{prop:regualrity of Log Brownian motion}
For $p>1$ there exists a continuous version of the $p$-log Brownian.
\end{prop}
\begin{proof}
	See \cite{MOCIOALCA_Viens2005}, Definition 18 and below, or \cite[Remark 2.5]{gulisashvili2020}. 
\end{proof}

\begin{cor}
\label{cor:Existence of Local time lBm} For $p>1$  the $d$-dimensional $p$-log Brownian motion $w^p$ is $\mathbb{P}-a.s.$ infinitely regularizing. 
\end{cor}

\begin{proof}
By definition, the $d$-dimensional $p$-lBm  is a centered Gaussian process, and according to Proposition~\ref{prop:regualrity of Log Brownian motion} it is continuous if $p>1$. 
By Theorem~\ref{thm: regualritry of avg op} together with Corollary~\ref{cor:cont in time of LT} we obtain that if $w^p$ is $\zeta$-LND for any  $\zeta>0$, then it is infinitely regularizing. So let us compute the conditional variance for $(s,t)\in \Delta_2$:
\begin{equation}
    \mathrm{Var}(w_t^p |\cF_s) =\int_s^tk(t-r)^2 dr I_d, 
\end{equation}
where $k(t)=|t\ln(1/t)^{2p}|^{-\frac{1}{2}}$ and $I_d$ is the $d$-dimensional unit matrix. 
 By elementary computations, using that $\frac{d}{dt}\frac{\ln(1/t)^{1-2p}}{2p-1}= k(t)^2$, we obtain that 
\begin{equation}
    \int_s^tk(t-r)^2dr=\left(2p-1\right)^{-1}\ln(1/(t-s))^{1-2p}. 
\end{equation}
Of course $\inf_{t\in (0,T]} \inf_{s \in [0,t)}\frac{|\ln(1/(t-s))|^{1-2p}}{(t-s)^{2\zeta}}>0$, so $w^p$ is $\zeta$-LND for any $\zeta > 0$ and therefore infinitely regularizing.

\end{proof}

\subsection{Infinite series of fractional Brownian motions}
We will here show that also the process considered in \cite{ProskeBanosAmine} is an infinitely regularizing process according to Definition \ref{def:infinitely-regularizing}.

\begin{prop}
Consider the process $\mathbb{B}:\left[0,T\right]\times \Omega \rightarrow\mathbb{R}^{d}$ introduced
in \cite{ProskeBanosAmine} given by 
\[
\mathbb{B}_{t}:=\sum_{n\geq0}\lambda_{n}B_{t}^{H_{n}}.
\]
Here $\left(\lambda_{n}\right)_{n\geq0}$ and and $\left(H_{n}\right)_{n\geq0}\in (0,1)$ are  null sequences such that   $\lambda_{n},H_n>0$ for all $n\geq0$. Moreover, 
$(B^{H_{n}})_{n\ge0}$ is sequence of independent $\mathbb{R}^{d}$-valued
fractional Brownian motion of Hurst parameter $H_{n}$. Additionally, we assume that  
\begin{equation}\label{eq:cont of monster}
\sum_{n\geq0}|\lambda_{n}|\mathbb{E}\left[\sup_{0\leq s\leq1}|B_{s}^{H_{n}}|\right]<\infty.
\end{equation}
 Then there exists a null set $\mathcal{N}\subset\Omega$ such that $\mathbb B(\omega)$ is infinitely regularizing for all $\omega\in \mathcal{N}^c$.
\end{prop}

\begin{proof}
We assume that $B^{H_n}$ is given as the Wiener-It\^{o} integral 
\begin{equation}
    B_t^{H_n}:=2H_n \int_{-\infty} ^t [(t-s)^{H_n-\frac{1}{2}}_+-(-s)^{H_n-\frac
    {1}{2}}_+]dB^n_r, 
\end{equation}
where $\left(B^n\right)_{n\in \mathbb{N}}$
 is a sequence of independent $\mathbb{R}^d$-valued two-sided Brownian motions and for convenience we chose the normalizing factor $2H_n$ instead of the usual $\Gamma(H_n + \frac12)^{-1}$. Since $\sum_{n\geq0}|\lambda_{n}|\mathbb{E}\left[\sup_{0\leq s\leq1}|B_{s}^{H_{n}}|\right]<\infty$ the process $\mathbb B$ is almost surely the uniform limit of continuous functions and therefore continuous itself. So to conclude the proof it suffices to show that $\mathbb B$ is $\zeta$-LND for any $\zeta>0$. The processes $B^{H_n}$ and $B^{H_m}$ are independent for $m\neq n$, and thus the conditional covariance is
\begin{equation}
    \mathrm{cov}(\BB_t|\cF_s)=\sum_{n\in \NN}  \lambda_n^2(t-s)^{2H_n} I_d. 
\end{equation}
Since $(H_n)_{n\in \NN }$ is a null sequence there exists $m$ such that $H_m < \zeta$, and then
\[
	\frac{\sum_n \lambda_n^2(t-s)^{2H_n}}{(t-s)^{2\zeta}} \ge \lambda_m^2 T^{2(H_m - \zeta)} > 0,
\]
where we used that $\lambda_m > 0$. This concludes the proof.
\end{proof}

\section{Path-wise construction of infinitely regularizing averaging operators}\label{sec:regularzing operators}

In this section we investigate the spatio-temporal regularity of ``averaging operators''. For a continuous path $w \in C([0,T], \mathbb R^d)$ and a measurable function $b:\mathbb R^d\rightarrow \mathbb{R}^d $, we define the averaging operator $T^{w}$ as
\begin{equation}\label{eq: avg op def}
T_{s,t}^{w}b\left(x\right):=\int_{s}^{t}b\left(x+w_{r}\right)dr.
\end{equation}
Such operators have previously 
been studied by Tao and Wright in \cite{TaoWri} in the case of deterministic
perturbations $w$, and more recently by Catellier and Gubinelli \cite{Catellier2016} in their study of the regularizing effect
of fractional Brownian motions on ODEs.

Our first result is that if $w$ is infinitely regularizing according to Definition \ref{def:infinitely-regularizing}, then the averaging operator $T^w$  can be uniquely extended to any $b \in \mathscr S'$.

Recall that $\cC^\alpha(\langle x \rangle^{-\kappa}))$ denotes the weighted Besov space of Definition \ref{def:Besov space}, with weight $\langle x\rangle^{-\kappa}:=(1+|x|^2)^{-\frac{\kappa}{2}}$

\begin{prop}
\label{prop:Average operator}
Let $w \in C([0,T],\mathbb R^d)$ be infinitely regularizing and let $b \in \mathscr S'$.  There exist  a $\kappa \in \mathbb R$ depending on $b$ and a unique function
\[
	T^wb \in \bigcap_{\substack{\gamma \in (0,1), \\ \alpha > 0}} C^\gamma([0,T], \mathcal C^\alpha(\langle x \rangle^{-\kappa}))
\]
such that $T_0^wb \equiv 0$, and with the property that for any sequence of continuous functions $(b_n)_{n\in \mathbb{N}} \subset C(\mathbb R^d) \cap \mathscr S'$ that converges to $b$ in $\mathscr S'$ we have
\[
	\lim_{n \to \infty} \| T^wb - T^wb_n \|_{C^\gamma([0,T],\mathcal C^\alpha(\langle x \rangle^{-\lambda}))} = 0
\]
for some $\lambda \in \mathbb R$ and all $\gamma \in (0,1)$, $\alpha > 0$.
\end{prop}

\begin{proof}
	If $b$ is continuous, then
	\[
		T_{s,t}^w b (x) = \int_s^t b(x + w_r) dr = \int_{\mathbb R^d} b(x+z) L_{s,t}(z) dz = \langle b, L_{s,t}(\cdot - x)\rangle,
	\]
	where $L$ is the local time associated to $w$. 
	Since $L_{s,t}(\cdot - x) \in C^\infty_c(\mathbb R^d)$, the right hand side makes sense for all $b \in \mathscr S'$ and we take it as the definition of $T^b_{s,t}(x)$.
	 To see the claimed regularity note that for $b \in \mathscr S'$ there exist $\kappa \in \mathbb R$ and $k \ge 0$ such that $|\langle b, \varphi \rangle| \lesssim \max_{|\alpha| \le k} \|\langle \cdot \rangle^{\kappa} \partial^\alpha \varphi \|_{\infty}$. In particular,
	\begin{align*}
		|T_{s,t}^b(x)| = |\langle b, L_{s,t}(\cdot - x)\rangle| & \lesssim \max_{|\alpha| \le k} \|\langle \cdot \rangle^{\kappa} \partial^\alpha L_{s,t}(\cdot - x) \|_{\infty} \\
		& \lesssim \sup_{z \in \mathbb R^d} \frac{\langle z \rangle^\kappa}{\langle z-x\rangle^\kappa} |t-s|^\gamma \| L \|_{C^\gamma([0,T],\mathcal C^{k+1}(\langle x \rangle^{\kappa}))} \\
		& \lesssim \langle x \rangle^\kappa |t-s|^\gamma,
	\end{align*}
	 where we used that  $\langle x\rangle =(1+|x|^2)^{\frac{1}{2}}$ and that $\| L \|_{C^\gamma([0,T],\mathcal C^{k+1}(\langle x \rangle^{\kappa}))} \simeq \| L \|_{C^\gamma([0,T],\mathcal C^{k+1})}$ as $L$ is compactly supported. To control the derivatives note that $T_{s,t}^wb$ is essentially a convolution, and thus $\partial^\beta T_{s,t}^w b(x) = \langle b, (-1)^\beta (\partial^\beta L_{s,t})(\cdot - x)\rangle$, from where the same arguments as above yield
\begin{equation}\label{eq:weighted derivatve}
		|\partial^\beta T_{s,t}^w b(x)| \lesssim  \langle x \rangle^\kappa |t-s|^\gamma,
\end{equation}
	and therefore $T^w b \in C^\gamma([0,T], \mathcal C^\alpha(\langle x \rangle^{-\kappa}))$ for all $\alpha > 0$ and all $\gamma \in (0,1)$.
	If a sequence of smooth functions $(b_n)_{n\in \mathbb{N}} \subset \mathscr S'$ converges to $b$ in $\mathscr S'$, then there exist $\lambda \in \mathbb R$ 
	and $\ell \ge 0$ such that $|\langle b_n, \varphi \rangle| \lesssim \max_{|\alpha| \le \ell} \|\langle \cdot \rangle^{\lambda} \partial^\alpha \varphi \|_{\infty}$ uniformly in $n$, 
	see \cite[Theorem~V.7]{Reed1980}. Therefore, the convergence of $T^w b_n$ to $T^w b$ in $C^\gamma([0,T], \mathcal C^\alpha(\langle x \rangle^{-\lambda}))$ follows as above.
\end{proof}

\begin{cor}\label{cor:Besov drift}
	If $b \in B^\beta_{p,q}$ for some $\beta \in \mathbb R$ and $p,q \in [1,\infty]$, then we have (without weights):
	\[
		T^w b \in \bigcap_{\substack{\gamma \in (0,1), \\ \alpha > 0}} C^\gamma([0,T], \mathcal C^\alpha).
	\]
\end{cor}

\begin{proof}
	If $b \in B^\beta_{p,q}$, then we have with the conjugate exponents $p',q'$ of $p,q$:
	\[
		|\langle b, L_{s,t}(\cdot - x)\rangle| \lesssim \| b \|_{B^\beta_{p,q}} \| L_{s,t}(\cdot - x) \|_{B^{-\beta}_{p',q'}} =  \| b \|_{B^\beta_{p,q}} \| L_{s,t} \|_{B^{-\beta}_{p',q'}} \lesssim \| L_{s,t} \|_{\mathcal C^{-\beta+\epsilon}},
	\]
	where in the last step we used that $L$ is compactly supported and therefore we can decrease the integrability index from  $\infty$ to $p'$ while only paying a constant, and that we can replace $q'$ by $\infty$ if we give up $\epsilon$ regularity. This shows that we can take $\kappa = 0$ in the proof (and then in the statement) of Proposition~\ref{prop:Average operator}.
\end{proof}

\section{Existence, uniqueness and flow differentiability of perturbed ODEs }\label{sec:ex a Uni}

We are now ready to apply the concept of averaging operators  to ODEs perturbed by noise. Formally,  we will consider 
the equation 
\begin{equation}
\tilde{y}_{t}^{x}=x+\int_{0}^{t}b\left(\tilde{y}^x_{r}\right)dr+w_{t},\qquad (t,x)\in[0,T] \times \mathbb{R}^{d},\label{eq:ode hat}
\end{equation}
for a Schwartz distribution $b$ and an infinitely regularizing continuous path $w$. To interpret this equation rigorously, we set $y_{t}^{x}:=\tilde{y}_{t}^{x}-w_{t},$
and observe that $y$ formally solves
\begin{equation}
y_{t}^{x}=x+\int_{0}^{t}b\left(y_{r}^x+w_{r}\right)dr,\qquad (t,x)\in [0,T] \times \mathbb{R}^{d}.\label{eq:our ode}
\end{equation}
To make sense of the integral on the right hand side we consider a sequence $(b_n)_{n\in \mathbb{N}}$ of continuous functions converging to $b\in \mathscr{S}^\prime$. Then, inspired by the construction of the operator $T^w b$ in Proposition \ref{prop:Average operator},  we will show that the following limit exists:
\begin{equation}
\int_{0}^{t}b\left(y_{r}+w_{r}\right)dr:=\lim_{n\rightarrow \mathbb{N}} \int_{0}^{t}b_n\left(y_{r}+w_{r}\right)dr.
\end{equation}
 To this end we use the non-linear Young integral of~\cite{Catellier2016}, for which we first give a simplified construction.

\subsection{Non-linear Young integration }

Let $\varXi:\Delta^T_2 \to \RR^d$ and consider the Riemann sum of $\Xi$ over a partition $\mathcal{P}$ of
a set $\left[s,t\right]\subset\left[0,T\right]$:
\[
\mathcal{I}_{\mathcal{P}}\left(\varXi\right)_{s,t}=\sum_{\left[u,v\right]\in\mathcal{P}}\varXi_{u,v}.
\]
The \emph{sewing lemma} (\cite[Proposition~1]{Gubinelli2004}, see also \cite[Lemma~4.2]{FriHai}) gives explicit conditions on the function $\varXi$ under which $\lim_{|\mathcal{P}|\rightarrow0}\mathcal{I}_{\mathcal{P}}\left(\varXi\right)$ exists. To state it, we first define the linear functional $\delta$ acting on
$f: \Delta^T_2 \to \RR^d$ as 
\begin{equation}
\delta_{u}f_{s,t}=f_{s,t}-f_{s,u}-f_{u,t}, \qquad (s,u,t) \in \Delta^T_3. \label{eq:elta}
\end{equation}

\begin{lem}[\cite{FriHai}, Lemma~4.2]
\label{lem:Non-Linear Young integral}Let $\beta\in\left(1,\infty\right)$, and let  $\varXi:\Delta_2^T\rightarrow\mathbb{R}^{d}$
be such that
\begin{align*}
\|\delta \varXi\|_\beta := \sup_{(s,u,t) \in \Delta^T_3} \frac{|\delta_{u}\varXi_{s,t}|}{|t-s|^{\beta}}<\infty.
\end{align*}
 Then there exists a unique function $\mathcal{I}(\varXi):[0,T]\to\mathbb{R}^d$ such that $\mathcal I(\varXi)_0=0$ and
\begin{equation}\label{eq:abs bound}
|\mathcal{I}\left(\varXi\right)_{s,t}-\varXi_{s,t}|\leq C\|\delta\varXi\|_{\beta}|t-s|^{\beta},
\end{equation}
where $C>0$ only depends on $\beta$ and $T$. Moreover, we have $\mathcal{I}\left(\varXi\right)_{0,t}=\lim_{|\mathcal{P}|\rightarrow0}\sum_{\left[u,v\right]\in\mathcal{P}}\varXi_{u,v}$. If additionally for $\alpha\in\left(0,1\right)$
\[
	\|\varXi\|_\alpha:=\sup_{(s,t) \in \Delta^T_2} \frac{|\varXi_{s,t}|}{|t-s|^{\alpha}}<\infty,
\]
then $\mathcal{I}(\varXi) \in C^\alpha([0,T],\mathbb{R}^d)$.
\end{lem}

Now let us consider again the integral in \eqref{eq:our ode}. 
If $b$ is continuous, then
\begin{align} \notag
\int_{0}^{t}b\left(y_{r}+w_{r}\right)dr & =\lim_{|\mathcal{P}|\rightarrow0}\sum_{\left[u,v\right]\in\mathcal{P}}b\left(y_{u}+w_{u}\right)\left(v-u\right)
\label{eq:gen young int}
=\lim_{|\mathcal{P}|\rightarrow0}\sum_{\left[u,v\right]\in\mathcal{P}}\int_{u}^{v}b\left(y_{u}+w_{r}\right)du\\
& = \lim_{|\mathcal{P}|\rightarrow0}\sum_{\left[u,v\right]\in\mathcal{P}}T^w _{v,u} b\left(y_{u}\right),
\end{align}
where $T^w b$ is the averaging operator from \eqref{eq: avg op def}. If $w$ is infinitely regularizing and $\mathcal P$ is a fixed partition, then by Proposition~\ref{prop:Average operator} the sum on the right hand side is well defined even if only $b \in \mathscr S'$. The existence of the limit as $|\mathcal P| \to 0$ will follow from the sewing lemma. Note that the limit is not exactly a Young integral, since $T^w b$ is non-linear in its spatial argument:
\[
T_{s,t}^{w}b \left(z+y\right)\neq T_{s,t}^{w}b\left(z\right)+T_{s,t}^{w}b\left(y\right).
\]
Therefore, we need a non-linear extension of the Young integral, which we recite from \cite{Catellier2016} in the following proposition.

\begin{prop}[See also \cite{Catellier2016}, Theorem~2.4 or \cite{Hu2017}, Proposition~2.4]
\label{prop:application of non-linear sew leemm}
Let $\beta, \gamma \in (0,1)$ be such that $\beta + \gamma > 1$. Let $y\in C^\beta_T := C^{\beta}_T\RR^d$  and let $Y:[0,T]\times \RR^d\rightarrow  \RR^d$ be such that
\begin{equation}\label{eq:cond for Y}
	|Y_{s,t}(x)|+|\nabla Y_{s,t}(x)| \le F(x) |t-s|^\gamma,\qquad(s,t) \in \Delta^2_T, x \in \RR^d,
\end{equation}
where $F$ is a locally bounded function.
Then with $\varXi_{s,t}=Y_{s,t}(y_s)$, the integral 
\begin{equation}
\int_{s}^{t}Y_{dr}\left(y_{r}\right):=\mathcal{I}\left(\varXi\right)_{s,t},\label{eq:def of non linear integral}
\end{equation}
 is well defined according to Lemma \ref{lem:Non-Linear Young integral}. Moreover, let $\tilde{Y}:[0,T]\times \RR^d\rightarrow \RR^d $ be another function satisfying \eqref{eq:cond for Y} (for a potentially different function $F$). Then for any $0<s<t\leq T$ the following bound holds
 \begin{equation}\label{eq:stability of NLY integral}
 \Big|\int_{s}^t Y_{d r}(y_r)-\int_{s}^t\tilde{Y}_{d r}(y_r)\Big|\leq  C(1+\|y\|_\beta)\|Y-\tilde{Y}\|_{\gamma,1,\mathcal B_{\|y\|_\infty}},
 \end{equation}
 where for $K\in \RR_+$ we define
\begin{equation}\label{norm like}
\|Y-\tilde{Y}\|_{\gamma,1,\mathcal B_K}:=\sup_{\substack{ s\neq t\in [0,T] \\ |x|\leq K}} \frac{|(Y_{s,t}-\tilde{Y}_{s,t})(x)|}{|t-s|^\gamma}+\sup_{\substack{ s\neq t\in [0,T] \\ |x|\leq K}} \frac{|\nabla (Y_{s,t}-\tilde{Y}_{s,t})(x)|}{|t-s|^\gamma}.
\end{equation}

\end{prop}

\begin{proof}
Set $\varXi_{s,t}=Y_{s,t}(y_s)$. 
Since $Y_{s,t}=Y_{t}-Y_{s}$, we have 
\begin{align*}
	|\delta_{u}\varXi_{s,t}|=|Y_{u,t}(y_{s})-Y_{u,t}(y_{u})|\leq  \sup_{|x| \le \|y\|_\infty} F(x) |t-u|^{\gamma}|y_{s,u}|\leq\sup_{|x| \le \|y\|_\infty} F(x)\|y\|_{C^\beta_T}|t-s|^{\gamma+\beta}.
\end{align*}
 Together with the bound for $Y$ in \eqref{eq:cond for Y}, it follows from Lemma \ref{lem:Non-Linear Young integral} that the integral $\cI(\varXi)$ is well defined. Let us now prove  the stability inequality in \eqref{eq:stability of NLY integral}. 
To this end, set $\tilde{\varXi}_{s,t}:=Y_{s,t}(y_s)-\tilde{Y}_{s,t}(y_s)$, and observe that 
\begin{equation}
|\tilde{\varXi}_{s,t}|\leq \|Y-\tilde{Y}\|_{\gamma,1,\mathcal B_{\|y\|_\infty}}|t-s|^\gamma.
\end{equation}
 Elementary computations similar to the ones above yield
\begin{equation*}
\delta_{u}\tilde{\varXi}_{s,t}= Y_{u,t}(y_s)-\tilde{Y}_{u,t}(y_s)-Y_{u,t}(y_u)+\tilde{Y}_{u,t}(y_u)
\end{equation*}
By the fundamental theorem of calculus, we have that 
\begin{equation}
Y_{u,t}(y_s)-\tilde{Y}_{u,t}(y_s)-Y_{u,t}(y_u)+\tilde{Y}_{u,t}(y_u)=\int_0^1\nabla (Y-\tilde{Y})_{u,t}(y_s+\theta y_{s,u})d \theta y_{s,u}.  
\end{equation}
It therefore follows directly that 
\begin{equation}
|\delta_{u}\tilde{\varXi}_{s,t}|\leq \|y\|_{\beta} \|Y-\tilde{Y}\|_{\gamma,1,\mathcal B_{\|y\|_\infty}}|t-s|^{\gamma+\beta}.
\end{equation}
We conclude again from Lemma \ref{lem:Non-Linear Young integral} that $\cI(\tilde{\varXi})$ is well defined, and from \eqref{eq:abs bound} we conclude that \eqref{eq:stability of NLY integral} holds. 

\end{proof}

\subsection{Abstract non-linear Young equations}

Here we use the non-linear Young integral from Proposition \ref{prop:application of non-linear sew leemm} to construct solutions to an abstract non-linear integral equation. Later we will apply these abstract results to our equation~\eqref{eq:our ode}. 
\begin{prop}
\label{lem:Local existence and uniqueness}Let $Y \in C^{0,1}([0,T]\times\mathbb{R}^{d},\mathbb{R}^{d})$
be such that for some $\gamma\in(\frac{1}{2},1)$
and $\delta>\frac{1}{\gamma}$ the following conditions hold for $s,t\in\left[0,T\right]$ and $x,y\in\mathbb{R}^{d}$:
\begin{align*}
&{\rm (i)} & |Y_{s,t}\left(x\right)|+|\nabla Y_{s,t}\left(x\right)|\leq &G(x)|t-s|^{\gamma},
\\
&{\rm (ii)}&  |\nabla Y_{s,t}\left(x\right)-\nabla Y_{s,t}\left(y\right)| \leq& F\left(x,y\right)|t-s|^{\gamma}|x-y|^{\delta-1},
\end{align*}
where $G: \RR^d \to \RR_+$ and $F:\mathbb{R}^{2d}\rightarrow\mathbb{R}_{+}$
are locally bounded functions. Then for all $x\in\mathbb{R}^{d}$
there is a maximal existence time $T^{*} = T^\ast(x) \in\left(0,T\right] \cup\{\infty\}$
and a unique solution $y\in C^{\gamma}([0,T^{*})\cap[0,T])$
to
\begin{equation}\label{eq:abs oDE}
y_{t}=x+\int_{0}^{t}Y_{dr}\left(y_{r}\right).
\end{equation}
Here the non-linear Young integral $\int_{0}^{t}Y_{dr}\left(y_{r}\right)$
is as in Proposition \ref{prop:application of non-linear sew leemm}.
If $T^{*}<\infty$, then $\lim_{t\rightarrow T^{*}} |y_{t}|=\infty$. Moreover, the map $x\mapsto T^\ast(x)^{-1}$ is locally bounded. If $G$ and $F$ are bounded, then $T^\ast = \infty$.
\end{prop}

\begin{proof}
This is quite standard and the result follows from an application of the non-linear
sewing lemma, Proposition~\ref{prop:application of non-linear sew leemm}, together with a Picard iteration. For completeness we include the arguments.

Let $\tau\in\left[0,T\right]$ and $\gamma^{\prime}\in\left(1-\gamma,\gamma\right)$ be such that $\gamma + \delta (1-\gamma') > 1$ (note that $\gamma + \delta(\gamma-1) = \delta \gamma > 1$, so this is possible). Let $z \in C_\tau^{\gamma'}$.
Define the increment $\varXi_{s,t}:=Y_{s,t}\left(z_{s}\right)$. Then we obtain from
{\rm(i)}:
\begin{equation}
\begin{aligned}
|\varXi_{s,t}| & \leq G(z_s) |t-s|^{\gamma} ,\label{eq:psi diff}\\
|\delta_{u}\varXi_{s,t}| & =|Y_{u,t}\left(z_{s}\right)-Y_{u,t}\left(z_{u}\right)|\leq \sup_{|a|\le \|z\|_\infty} G(a) \|z\|_{C_{\tau}^{\gamma^{\prime}}}|t-s|^{\gamma+\gamma^{\prime}},
\end{aligned}
\end{equation}
where $C_{\tau}^{\gamma^{\prime}} = C^{\gamma^{\prime}}\left(\left[0,\tau\right]\right).$ Since $\gamma+\gamma^{\prime}>1$ it follows from Lemma \ref{lem:Non-Linear Young integral}
that the map
\begin{gather*}
\Gamma :\left\{ z\in C^{\gamma^{\prime}}(\left[0,\tau\right],\mathbb{R}^{d})\big|z_{0}=x\right\} \rightarrow\left\{ z\in C^{\gamma^{\prime}}(\left[0,\tau\right],\mathbb{R}^{d})\big|z_0=x\right\}, \\
\Gamma\left(z\right)_{t}  =x+\int_{0}^{t}Y_{dr}\left(z_{r}\right)
\end{gather*}
is well defined and satisfies
\begin{align*}
|\Gamma(z)_{s,t}| & \leq\left|\int_{s}^{t}Y_{dr}\left(z_{r}\right)-Y_{s,t}\left(z_{s}\right)\right|+|Y_{s,t}\left(z_{s}\right)|\\
 & \lesssim|t-s|^{\gamma+\gamma^{\prime}} \sup_{|a| \le \|z\|_\infty} G(a)\|z\|_{C_{\tau}^{\gamma^{\prime}}}+\sup_{|a|\le \|z\|_\infty} G(a)|t-s|^{\gamma}\\
 & \lesssim\tau^{\gamma-\gamma^{\prime}}|t-s|^{\gamma^{\prime}} \sup_{|a|\le \|z \|_\infty} \left( G\left(a\right)\|z\|_{C_{\tau}^{\gamma^{\prime}}}+G(a) \right).
\end{align*}
This implies that for sufficiently small $\tau>0$ (depending on $|x|$)
the map $\Gamma$ leaves the ball 
\[
\mathcal{B}_{2|x|}=\left\{ z\in C_{\tau}^{\gamma^{\prime}}\big|z_{0}=x,\, \|z\|_\infty \vee \|z\|_{C_{\tau}^{\gamma^{\prime}}}\leq2|x|\right\} 
\]
invariant. Moreover, for two paths $z,\tilde{z}\in\mathcal{B}_{2|x|}$ we have
\[
|\Gamma\left(z\right)_{s,t}-\Gamma\left(\tilde{z}\right)_{s,t}|\leq \left|Y_{s,t}\left(z_{s}\right)-Y_{s,t}\left(\tilde{z}_{s}\right)\right| +\left|\int_{s}^{t}\left[Y_{dr}\left(z_{r}\right)-Y_{dr}\left(\tilde{z}_{r}\right)\right]-\left[Y_{s,t}\left(z_{s}\right)-Y_{s,t}\left(\tilde{z}_{s}\right)\right]\right|.
\]
For $u\in\left[s,t\right]$ we rewrite
\begin{equation}\label{delta Y inc}
\begin{aligned}
\delta_{u}\left[Y_{s,t}\left(z_{s}\right)-Y_{s,t}\left(\tilde{z}_{s}\right)\right]&= (Y_{u,t}(z_{s})-Y_{u,t}(\tilde{z}_{s})) -\left(Y_{u,t}\left(z_{u}\right)-Y_{u,t}\left(\tilde{z}_{u}\right)\right)
\\
&=\int_{0}^{1}\nabla Y_{u,t}\left(\tilde{z}_{s}+\lambda\left(z_{s}-\tilde{z}_{s}\right)\right) \cdot \left(z_{s}-\tilde{z}_{s}\right) d\lambda 
\\
&\quad -\int_{0}^{1}\nabla Y_{u,t}\left(\tilde{z}_{u}+\lambda\left(z_{u}-\tilde{z}_{u}\right)\right) \cdot \left(z_{u}-\tilde{z}_{u}\right) d\lambda
\end{aligned}
\end{equation}
Invoking condition ${\rm (ii)}$ on the function $Y$, we observe that 
\begin{align*}
 &|\delta_{u}\left[Y_{s,t}\left(z_{s}\right)-Y_{s,t}\left(\tilde{z}_{s}\right)\right]|
 \\
  &\leq \left|\int_{0}^{1} \left[ \nabla Y_{u,t}\left(\tilde{z}_{s}+\lambda\left(z_{s}-\tilde{z}_{s}\right)\right)- \nabla Y_{u,t}\left(\tilde{z}_{u}+\lambda\left(z_{u}-\tilde{z}_{u}\right)\right)\right] \cdot \left(z_{s}-\tilde{z}_{s}\right) d\lambda\right|
  \\
  &\quad + \left|\int_{0}^{1}\nabla Y_{u,t}\left(\tilde{z}_{s}+\lambda\left(z_{s}-\tilde{z}_{s}\right)\right) \cdot \left(z_{s}-\tilde{z}_{s}-z_{u}-\tilde{z}_{u}\right) d\lambda\right|
  \\
  & \leq  \int_{0}^{1}F\left(\tilde{z}_{s}+\lambda\left(z_{s}-\tilde{z}_{s}\right),\tilde{z}_{u}+\lambda\left(z_{u}-\tilde{z}_{u}\right)\right)|t-u|^{\gamma}
  \\
  &\hspace{30pt} \times  |\tilde{z}_{s}+\lambda\left(z_{s}-\tilde{z_{s}}\right)-\left(\tilde{z}_{u}+\lambda\left(z_{u}-\tilde{z_{u}}\right)\right)|^{\delta-1}d\lambda\|z-\tilde{z}\|_{\infty}\\
  &\quad +  \int_{0}^{1} G(\tilde{z}_{s}+\lambda(z_{s}-\tilde{z}_{s})) |t-u|^{\gamma} d\lambda|t-s|^{\gamma^{\prime}}\|z-\tilde{z}\|_{C_{\tau}^{\gamma^{\prime}}}\\
 & \lesssim  \|F\|_{\mathcal{B}_{2|x|}}|t-s|^{\gamma+\gamma^{\prime}\left(\delta-1\right)}|2x|^{\delta-1}\|z-\tilde{z}\|_{C_{\tau}^{\gamma^{\prime}}}
   +|t-s|^{\gamma+\gamma^{\prime}}\|G\|_{\mathcal{B}_{2|x|}}\|z-\tilde{z}\|_{C_{\tau}^{\gamma^{\prime}}},
\end{align*}
where
\[
	\|G\|_{\mathcal{B}_{2|x|}} = \sup_{|a| \le 2|x|} G(a) ,\qquad\|F\|_{\mathcal{B}_{2|x|}} = \sup_{|a|,|b| \le 2|x|} F(a,b).
\]
Recall that $\gamma+\left(\delta-1\right)\gamma^{\prime}>1$. Furthermore, the bound in ${\rm (i)}$ gives 
\[
|Y_{s,t}\left(z_{s}\right)-Y_{s,t}\left(\tilde{z}_{s}\right)|\leq\|G\|_{\mathcal{B}_{2|x|}}|t-s|^{\gamma}\|z-\tilde{z}\|_{\infty}\le \|G\|_{\mathcal{B}_{2|x|}}\tau^{\gamma^{\prime}}|t-s|^{\gamma}\|z-\tilde{z}\|_{C_{\tau}^{\gamma^{\prime}},}
\]
where we used that $z_{0}=\tilde{z}_{0}=x$, and therefore
$
\|z-\tilde{z}\|_{\infty} \le \tau^{\gamma'}\|z-\tilde{z}\|_{C_{\tau}^{\gamma'}}.
$
So after possibly further decreasing $\tau>0$, depending on $|x|$, we get
a contraction on $\mathcal{B}_{2|x|}$. Since the maximum possible choice for $\tau$ only depends on $|x|$ and it is bounded away from $0$ if $|x|$ is bounded, we can choose $\tau(x)$ such that the map $x \mapsto \tau(x)^{-1}$ is locally bounded.

 Moreover, it is a simple exercise
to check that for $z\in C_{\tau}^{\gamma'}$ we have $\Gamma\left(z\right)\in C_{\tau}^{\gamma}$, so the unique fixed point $\left(y_{t}\right)_{t\in\left[0,T\right]}$
is even $\gamma$-H\"older continuous. Now we can iterate the construction
and extend the solution to $\left[0,\tau+\tau^{\prime}\right]$ for
some $\tau^{\prime}\leq\tau$, etc. We just showed that $\tau(x)$ only depends on the size
of the initial condition $|x|$, so if we had $\sup_{t\in\left[0,T^{*}\right)}|y_{t}|<\infty$
and $T^{*}<T$, then we could extend the solution beyond $T^{*}$
and thus $T^{*}$ could not have been the maximal time of existence. Since $T^\ast > \tau$ the local boundedness of $x \mapsto T^\ast(x)^{-1}$ follows from that of $x \mapsto \tau(x)^{-1}$.

If $F$ and $G$ are bounded, then there exists a fixed $\tau>0$ such that for any starting point $x$ the map $\Gamma$ leaves the ball $\mathcal B_{2|x|}$ invariant. Therefore, the solution $y$ with initial value $x$ satisfies  $\sup_{t \in [0,\sigma]}|y_\sigma| \le 2^{\left\lceil \frac{\sigma}{\tau} \right\rceil} |x|$ on any interval $[0,\sigma] \subset [0,T]$, and it does not explode in finite time, i.e. $T^\ast = \infty$.
\end{proof}

\subsection{Application to perturbed ODEs}

We will now apply the abstract results from the previous section to define solutions to Equation \eqref{eq:our ode} and to prove their existence and uniqueness and the smoothness of the associated flow.

\begin{lem}
	Let $w$ be infinitely regularizing, let $b \in \mathscr S'$, and let $T^w b$ the averaging operator defined in \eqref{eq: avg op def}. Then for all $\epsilon >0$ and all $y\in C^{\epsilon}([0,T],\RR^d)$ the non-linear Young integral $\int_{0}^{t}T_{dr}^{w} b(y_{r})$ is well defined.
\end{lem}

\begin{proof}
	By Proposition~\ref{prop:Average operator} the function $Y_{t}(x) = T^w_t b(x)$ satisfies $|\nabla Y_{s,t}(x)| \lesssim \langle x \rangle^\kappa |t-s|^\gamma$ for some $\kappa \in \RR$ and for all $\gamma < 1$. In particular we can choose $\gamma > 1 - \epsilon$, and then the claim follows from Proposition~\ref{prop:application of non-linear sew leemm}.
\end{proof}

\begin{defn}\label{def:perturbed-ode}
	Let $w$ be infinitely regularizing, let $b \in \mathscr S'$ and let $\tau \le T$ and $\tilde y \in C([0,\tau],\RR^d)$. Then we say that $\tilde y$ solves the equation
	\begin{equation}\label{eq:perturbed ode def naiv}
		\tilde y_t = x + \int_0^t b(\tilde y_r) dr + w_t,
	\end{equation}
	if $y = \tilde y - w$ is in $C^\epsilon([0,\tau],\RR^d)$ for some $\epsilon>0$ and
	\begin{equation}\label{eq:perturbed ode def rigorous}
		y_t = x + \int_0^t T_{dr}^wb (y_r),\qquad t \in [0,\tau].
	\end{equation}
\end{defn}

\begin{lem}
	Let $w$ be infinitely regularizing, let $b \in \mathscr S'$ and let $\tau \le T$ and $\tilde y \in C([0,\tau],\RR^d)$ be such that $y = \tilde y - w$ is in $C^\epsilon([0,\tau],\RR^d)$ for some $\epsilon>0$. Then $\tilde y$ solves~\eqref{eq:perturbed ode def naiv} if and only if for any sequence $(b_n)\subset C(\RR^d) \cap \mathscr S'$ converging to $b$ in $\mathscr S'$ we have
	\[
		\tilde y_t = x + \lim_{n \to \infty} \int_0^t b_n(\tilde y_r) dr + w_t,\qquad t \in [0,\tau].
	\]
\end{lem}

\begin{proof}
	By the convergence result for the averaging operator in Proposition~\ref{prop:Average operator} together with continuity properties of the non-linear Young integral which follow  directly from \eqref{eq:stability of NLY integral} in Proposition~\ref{prop:application of non-linear sew leemm} we have $\int_0^t T_{dr}^wb (y_r) = \lim_{n \to \infty} \int_0^t T^w_{dr} b_n(y_r)$ for $t \in [0,\tau]$. Therefore, the claim follows from~\eqref{eq:gen young int}.
\end{proof}

\begin{lem}\label{lem:Local existence and uniqueness perturbed ode}Let $w$ be infinitely regularizing, let $b \in \mathscr S'$ and let $\gamma \in (\frac12,1)$.  For all $x\in\mathbb{R}^{d}$
there exists a maximal existence time $T^\ast = T^{*}(x)\in\left(0,T\right] \cup\{\infty\}$
and a unique solution $y\in C^{\gamma}([0,T^{*})\cap[0,T])$
to
\begin{equation}\label{eq:abs oDE}
y_{t}=x+\int_{0}^{t}T_{dr}^w b(y_r),\qquad t \in [0,T^\ast) \cap [0,T].
\end{equation}
If $T^{*}<\infty$, then $\lim_{t\rightarrow T^{*}} |y_{t}|=\infty$. Moreover, the map $x\mapsto T^\ast(x)^{-1}$ is locally bounded. If $b \in B^\beta_{p,q}$ for some $\beta \in \RR$ and $p,q \in [1,\infty]$, then $T^\ast = \infty$.
\end{lem}

\begin{proof}
	According to Proposition~\ref{prop:Average operator} there exists $\kappa \in \RR$ such that $Y_t(x) = T^w_t b(x)$ is in $C^\gamma_T \cC^\alpha(\langle x \rangle^{-\kappa})$ for all $\gamma < 1$ and $\alpha >0$. In particular it satisfies the assumptions of Proposition~\ref{lem:Local existence and uniqueness}. If $b \in B^\beta_{p,q}$, then $Y \in C^\gamma_T \cC^\alpha$ (without weight), and therefore the global existence follows from the last part of Proposition~\ref{lem:Local existence and uniqueness}.
\end{proof}

To complete the proof of Theorem~\ref{thm:main} we have to show the differentiability of the flow $x \mapsto y^x$, where $y^x$ solves the equation with $y^x_0 = x$. In order to obtain higher order  differentiability of the flow, we extend the results of \cite{Hu2017} where they prove the flow property of the solution map of non-linear Young equations, as well as its first order Frech\'et differentiability.   We achieve this by solving the equation for $(y^x, \nabla y^x, \dots, \nabla^k y^x)$, whose explosion time a priori might depend on $k$. To show that it is independent of $k$ and that the flow exists as long as $y^x$ stays bounded, we introduce an abstract notion:

\begin{defn}
\label{def:lower triangular structure}Let $k\geq0$ and $d=d_{0}+\ldots+d_{k}$
and let $Y\in C^{0,1}([0,T]\times\mathbb{R}^{d},\mathbb{R}^{d})$
be of the form 
\begin{equation}\label{eq:rep of Y}
Y_{t}\left(z\right)=[Y_{t}^{0}(z^{0}),Y_{t}^{1}(z^{\leq1}),\ldots,Y_{t}^{k}(z^{\leq k})]
\end{equation}
 for all $z=(z^{0},z^{1},\ldots,z^{k})\in\mathbb{R}^{d_{0}+\ldots+d_{k}}$, where $z^{\leq \ell}:=(z^0 ,\ldots,z^\ell)$.
Let $\gamma\in(\frac{1}{2},1)$, $\delta>\frac{1}{\gamma}$ and assume that $Y^0$ satisfies the condition of Proposition~\ref{lem:Local existence and uniqueness}, while each of the components $Y^\ell$ for $\ell \in \{1,\dots, k\}$ satisfies the following three bounds:
\begin{align*}
&{\rm (i)}& |Y_{s,t}^{\ell}(z^{\leq \ell})| & \leq G_{\ell}(z^{\leq\ell-1})(1+|z^{\ell}|)|t-s|^{\gamma},
\\
&{\rm (ii)}& |Y_{s,t}^{\ell}(z^{\leq \ell})-Y_{s,t}^{\ell}(\tilde{z}^{\leq \ell})| & \leq|t-s|^{\gamma}H_{\ell}(z^{\leq\ell-1},\tilde{z}^{\leq\ell-1})
\\
&\,& & \hspace{30pt} \times (|z^{\ell}-\tilde{z}^{\ell}|+|z^{\ell}|\times|z^{\leq\ell-1}-\tilde{z}^{\leq\ell-1}|),
\\
&{\rm (iii)}& |\nabla Y_{s,t}^{\ell}(z^{\leq \ell})-\nabla Y_{s,t}^{\ell}(\tilde{z}^{\leq \ell})| & \leq F_\ell(z^{\leq \ell},\tilde{z}^{\leq \ell})|t-s|^{\gamma}|z^{\leq \ell}-\tilde{z}^{\leq \ell}|^{\delta-1},
\end{align*}
 where the functions $G_{\ell}$, $H_{\ell}$ and $F_\ell$ are positive and locally bounded. Then we say that $Y$ has a \emph{lower triangular structure. }
\end{defn}

If $Y$ has a lower triangular structure, then the maximal existence time of $y_t = x + \int_0^t Y_{dr} (y_r)$ is equal to the explosion time of $y^0$:

\begin{lem}
\label{lem:Triangular structure of y}Assume that $Y$ has a lower
triangular structure and let $y$ be the solution to $y_t = x + \int_0^t Y_{dr}(y_r)$, constructed in Proposition~\ref{lem:Local existence and uniqueness}, with maximal existence time $T^\ast \in (0,T] \cup \{\infty\}$. If $T^\ast < \infty$, then $\lim_{t \to T^\ast} |y^0_t| = \infty$.
\end{lem}

\begin{proof}
By definition $Y$ satisfies the conditions of Proposition~\ref{lem:Local existence and uniqueness}, so $y$ exists. Assume that $T^\ast < \infty$ and that $\sup_{t < T^\ast} |y^0_t| = C < \infty$. We claim that then also $\sup_{t < T^\ast} |y^{\le \ell}_t| < \infty$ for all $\ell \le k$, which is a contradiction to  the fact that $\sup_{t < T^\ast} |y_t| = \infty$ by Proposition~\ref{lem:Local existence and uniqueness}.

Assume that the claim holds for $\ell-1$ and let us show that then it also holds for $\ell$. Because of the lower triangular structure, $y^{\le \ell - 1}$ solves a non-linear Young equation with non-linearity $Y^{\le \ell - 1}$, and since $\sup_{t < T^\ast} |y^{\le \ell-1}_t| < \infty$ we deduce from Proposition~\ref{lem:Local existence and uniqueness} that also $\sup_{t < T^\ast} \| y^{\le \ell - 1}\|_{C_t^\gamma} < \infty$.

To obtain a bound for $|y^\ell|$ let $\varXi_{s,t}^{\ell}=Y^{\ell}_{s,t}\left(y_{s}^{\leq\ell}\right)$. Then there exists a constant $C>0$ which depends on $\sup_{t < T^\ast} |y^{\le \ell-1}_t|$ and $\sup_{t < T^\ast} \| y^{\le \ell - 1}\|_{C_t^\gamma}$ such that for any $s<t\in [0,\tau]\subset [0,T]$ 
\begin{align}
|\varXi_{s,t}^{\ell}| & \leq G_{\ell}(y_{s}^{\leq\ell-1})(1+|y_{s}^{\ell}|)|t-s|^{\gamma} \leq |t-s|^{\gamma} C \left(1+|y_{0}^{\ell}|+\tau^{\gamma}\|y^{\ell}\|_{C_{\tau}^{\gamma}}\right), \label{eq:y1}
\\
|\delta_{u}\varXi_{s,t}^{\ell}| & =|Y_{u,t}(y_{u}^{\leq\ell})-Y_{u,t}(y_{s}^{\leq\ell})|\nonumber 
\\
 & \leq|t-s|^{\gamma}H_{\ell}(y_{s}^{\leq\ell-1},y_{u}^{\leq\ell-1})(|y_{u}^{\ell}-y_{s}^{\ell}|+|y_{s}^{\ell}|\times|y_{u}^{\leq\ell-1}-y_{s}^{\leq\ell-1}|)\nonumber \\
 & \leq |t-s|^{2\gamma} C \left(\|y^{\ell}\|_{C_{\tau}^{\gamma}}+|y_{0}^{\ell}|+\tau^{\gamma}\|y^{\ell}\|_{C_{\tau}^{\gamma}}\right).\label{eq:del y}
\end{align}
With the help of these bounds we obtain from the sewing lemma (Lemma~\ref{lem:Non-Linear Young integral}) that for $\gamma'<\gamma$: 
\begin{align*}
|y_{s,t}^{\ell}| & =\left|\int_{s}^{t}Y^\ell_{dr}(y^{\le \ell}_{r})\right|\\
 & \lesssim \tau^{\gamma - \gamma'} |t-s|^{\gamma'} C \left(1+|y_{0}^{\ell}|+\tau^{\gamma^{\prime}}\|y^{\ell}\|_{C_{\tau}^{\gamma^{\prime}}}\right) \\
 & \quad +\tau^\gamma |t-s|^{\gamma^{\prime}} C \left(\|y^{\ell}\|_{C_{\tau}^{\gamma^{\prime}}}+|y_{0}^{\ell}|+\tau^{\gamma^{\prime}}\|y^{\ell}\|_{C_{\tau}^{\gamma^{\prime}}}\right).
\end{align*}
Therefore, there exists $\tau > 0$ which only depends on $C$ and $\gamma,\gamma'$ such that
\[
	\|y^\ell\|_{C^{\gamma'}_\tau} \vee \sup_{t \le \tau} | y^\ell_t | \le 2|y_0^\ell|.
\]
Since $\tau$ is fixed and does not depend on $y^\ell_0$ we  deduce that $\sup_{t < T^\ast} |y^\ell_t| \le 2^{\left\lceil \frac{T^\ast}{\tau} \right\rceil} |y^\ell_0| < \infty$ and this concludes the proof.
\end{proof}

Now we are ready to prove Theorem~2:

\begin{thm*}[Theorem~\ref{thm:main}]
Let $b\in\mathscr{S}^{\prime}$ be a Schwartz
distribution, and consider an infinitely regularizing path $w:\left[0,T\right]\rightarrow\mathbb{R}^{d}$
as in Definition \ref{def:infinitely-regularizing}. 
Then for all $x \in \mathbb R^d$ there exists $T^\ast = T^\ast(x) \in (0,T] \cup \{\infty\}$ such that there is a unique solution to the equation 
\[
y_{t}^{x}=x+\int_{0}^{t}b\left(y_{r}^x\right)dr+w_{t},
\]
 in $C\left(\left[0,T^\ast\right) \cap [0,T],\mathbb{R}^{d}\right)$. For $T^\ast(x) < \infty$ we have $\lim_{t \uparrow T^\ast(x)} |y^x_t| = \infty$. Moreover, the map $x \mapsto T^\ast(x)^{-1}$ is locally bounded, and if $\tau < T^\ast(x)$ for all $x \in U$ with an open set $U$, then the flow mapping $U \ni x \mapsto y^x_\cdot \in C([0,\tau],\mathbb R^d)$ is infinitely Fr\'echet differentiable.
\end{thm*}

\begin{proof}
It remains to prove the smoothness of the flow. Let $k \in \NN$. We define
\begin{align*}
{Y}_{s,t}^{\ell}(z^{\leq\ell}) & =\sum_{j=1}^{\ell}\sum_{i_{1}+\ldots+i_{j}=\ell}\nabla^{j}T_{s,t}^{w}b(z^{0})\left(z^{i_{1}}\otimes\ldots\otimes z^{i_{j}}\right),\qquad \textit{for}\qquad 0\leq\ell\leq k.
\end{align*}
Since $Y^\ell$ is an affine function of $z^\ell$ it is not hard to see that ${Y}=({Y}^{1},\ldots,{Y}^{k})$
has a lower triangular structure. Let now $x \in \RR^d$, let $U$ be an open neighborhood of $x$ and let $\tau \in [0,T]$ be such that $T^\ast(z)> \tau$  for all $x' \in U$. For $x' \in U$ assume that $z^{x'} = (z^{0,x'},\dots, z^{k,x'})$ solves
\[
	z^{x'}_t = \chi + \int_0^t Y_{dr}(z^{x'}_r)
\]
on the maximum existence interval, where
\[
	\chi = (x', I_d, 0, \dots, 0)
\]
for the $d$-dimensional unit matrix $I_d$. Then $z^{0,x'} = y^{x'}$ by definition of $Y$, and therefore Lemma~\ref{lem:Triangular structure of y} shows that $z^{x'}$ exists on $[0,\tau]$. We claim that $z^{\ell,x} = \nabla^\ell y^x$, as a Fr\'echet derivative in $C([0,\tau], \RR^d)$ equipped with the uniform norm. Below we prove this for $\ell = 1$, the general case is similar but the notation becomes more involved.

Before we prove the first order differentiability we first show local Lipschitz continuity. So let $x' \in U$ and define the integrand $\varXi_{s,t} = T^w_{s,t} b(y^{x'}_s) - T^w_{s,t} b(y^{x}_s)$. Then $y^x - y^{x'}$ is the sewing of $\varXi$. There exists a constant $C$ that depends on $y^x$ and $y^{x'}$ such that for $0\le s \le t \le \sigma \le \tau$:
\begin{align*}
	|\varXi_{s,t}| & \le C |t-s|^\gamma \sup_{t \le \sigma}| y^{x'}_t - y^x_t | \le C  |t-s|^\gamma (|x'-x| + \sigma^{\gamma'} \| y^{x'} - y^x\|_{C^{\gamma'}_\sigma}) \\
	|\delta_u \varXi_{s,t}| & = | (T^w_{u,t} b(y^{x'}_s) -  T^w_{u,t} b(y^{x'}_u)) - (T^w_{u,t} b(y^{x}_s) - T^w_{u,t} b(y^{x}_u))| \\
	& \le C |t-s|^{\gamma+\gamma'} \| y^{x'} - y^x\|_{C^{\gamma'}_\sigma}.
\end{align*}
So by the sewing lemma (Lemma~\ref{lem:Non-Linear Young integral}) we get for a new $C>0$:
\[
	|y^x_{s,t} - y^{x'}_{s,t}| \le C |t-s|^{\gamma'} \sigma^{\gamma - \gamma'} (|x' - x| + \|y^{x'} - y^x\|_{C^{\gamma'}_\sigma}),
\]
and therefore we have for sufficiently small $\sigma$ (depending only on $C$):
\[
	\|y^{x'} - y^x\|_{C^{\gamma'}_\sigma} \vee \sup_{t \le \sigma} | y^{x'}_t - y^x_t| \le 2 |x'-x|,
\]
and then iteratively
\[
	\|y^{x'} - y^x\|_{C^{\gamma'}_\tau} \vee \| y^{x'} - y^x \|_\infty \lesssim |x'-x|,
\]
which proves the local Lipschitz continuity.

Next we want to show that $z^{1,x}$ is the Fr\'echet derivative of $y^x$ in $x$. For that purpose we define the new integrand
\begin{align*}
	\varXi_{s,t} = T^w_{s,t} b(y^{x'}_s) - T^w_{s,t} b(y^{x}_s) - \nabla T^w_{s,t}(y^x_s) z^{1,x}_s (x'-x).
\end{align*}
Then $y^x - y^{x'} - z^{1,x}(x'-x)$ is the sewing of $\varXi$. There exists a $C>0$ that depends on $y^x, y^{x'}, z^{1,x}$, such that for $0\le s \le t \le \sigma \le \tau$:

\begin{align*}
	|\varXi_{s,t}| & \le C |t-s|^\gamma \left(\| y^{x'} - y^x \|_\infty^2 + \sup_{r \le \sigma} | y^{x'}_r - y^x_r - z^{1,x}_r (x'-x) | \right) \\
	& \lesssim C  |t-s|^\gamma \left(|x'-x|^2 + \sigma^{\gamma'} \| y^{x'} - y^x - z^{1,x} (x'-x)\|_{C^{\gamma'}_\sigma} \right) \\
	|\delta_u \varXi_{s,t}| & = \Big| (T^w_{u,t} b(y^{x'}_s) -  T^w_{u,t} b(y^{x'}_u)) - (T^w_{u,t} b(y^{x}_s) - T^w_{u,t} b(y^{x}_u))  \\
	& \hspace{30pt} - (\nabla T^w_{u,t}(y^{x}_s) z^{1,x}_s (x'-x) - \nabla T^w_{u,t}(y^{x}_u) z^{1,x}_u (x'-x))\Big| \\
	& \le C |t-s|^{\gamma+\gamma'} \left(|x-x|^2 + \| y^{x'} - y^x - z^{1,x} (x'-x)\|_{C^{\gamma'}_\sigma}\right).
\end{align*}
From here we obtain as before that
\[
	\| y^{x'} - y^x - z^{1,x} (x'-x)\|_{C^{\gamma'}_\sigma} \vee \| y^{x'} - y^x - z^{1,x} (x'-x)\|_\infty \lesssim |x - x'|^2,
\]
and therefore $z^{1,x}$ is indeed the Fr\'echet derivative of $y^x$.

So far we showed that $y^x$ is $k$ times Fr\'echet differentiable, but since $k$ was arbitrary $y^x$ is infinitely Fr\'echet differentiable as claimed.
\end{proof}

\bibliographystyle{plain}

\end{document}